\documentclass[12pt]{aptpub}
\usepackage{color}

\usepackage{amsmath,amstext,url}

\numberwithin{equation}{section} 

\makeatletter \@addtoreset{equation}{section}

\makeatletter \@addtoreset{lemma}{section}

\makeatletter \@addtoreset{theorem}{section}

\makeatletter \@addtoreset{proposition}{section}

\makeatletter \@addtoreset{corollary}{section}

\makeatletter \@addtoreset{remark}{section}

\makeatletter \@addtoreset{definition}{section}

\makeatletter \@addtoreset{example}{section}

\begin{document}

\thispagestyle{firstpg}

\vspace*{1.5pc} \noindent \normalsize\textbf{\Large {$n$-type Markov
Branching Processes with Immigration }} \hfill \vspace{0.5cm}

\hspace*{0.75pc}{\small\textrm{\uppercase{Li Junping
}}}\hspace{-2pt}$^{*}$ {\small\textit{Central South University }}

\hspace*{0.75pc}{\small\textrm{\uppercase{Wang Juan
}}}\hspace{-2pt}$^{**}$ {\small\textit{Central South University }}

\hspace*{0.75pc}{\small\textrm{\uppercase{Zang Yanchao
}}}\hspace{-2pt}$^{***}$ {\small\textit{Central South University }}


\par
\footnote{\hspace*{-0.75pc}$^{*}\,$Postal address:
 School of Mathematics and Statistics, Central
South University, Changsha, 410075, China. E-mail address:\
jpli@mail.csu.edu.cn }

\par
\footnote{\hspace*{-0.75pc}$^{**}\,$Postal address:
 School of Mathematics and Statistics, Central
South University, Changsha, 410075, China. E-mail address:\
wangjuanrose@126.com}

\par
\footnote{\hspace*{-0.75pc}$^{***}\,$Postal address:
 School of Mathematics and Statistics, Central
South University, Changsha, 410075, China. E-mail address:\
zycmail@126.com }

\par

\renewenvironment{abstract}{%
\vspace{2pt} \hspace*{2.25pc}
\begin{minipage}{14cm}
\footnotesize
{\bf Abstract}\\[1ex]
\hspace*{0.5pc}} {\end{minipage}}
\begin{abstract}
In this paper, we consider $n$-type Markov branching processes with
immigration and resurrection. The uniqueness criteria are first
established. Then, a new method is found and the explicit expression
of extinction probability is successfully obtained in the absorption
case, the mean extinction time is also given. The recurrence and
ergodicity criteria are given if the state ${\bf 0}$ is not
absorptive. Finally, if the resurrection rates are same as the
immigration rates, the branching property and decay property are
discussed in detail, it is shown that the process is a
superimposition of a $n$-type branching process and an immigration.
The exact value of the decay parameter $\lambda_Z$ is given for the
irreducible class ${\bf Z}_+^n$. Moreover, the corresponding
$\lambda_Z$-invariant measures/vectors and quasi-distributions are
presented.
\end{abstract}
\par
\vspace*{12pt} \hspace*{2.25pc}
\parbox[b]{26.75pc}{{
}} {\footnotesize {\bf Keywords:} $n$-type Markov branching process,
 immigration, recurrence, \\
 \hspace*{8.25pc}
   branching property, decay parameter, invariant measures/vectors
}
\par
\normalsize

\renewcommand{\amsprimary}[1]{
     \vspace*{8pt}
     \hspace*{2.25pc}
     \parbox[b]{24.75pc}{\scriptsize
    AMS 2000 Subject Classification: Primary 60J27 Secondary 60J35
     {\uppercase{#1}}}\par\normalsize}
\renewcommand{\ams}[2]{
     \vspace*{8pt}
     \hspace*{2.25pc}
     \parbox[b]{24.75pc}{\scriptsize
     AMS 2000 SUBJECT CLASSIFICATION: PRIMARY
     {\uppercase{#1}}\\ \phantom{
     AMS 2000
     SUBJECT CLASSIFICATION:
     }
    SECONDARY
 {\uppercase{#2}}}\par\normalsize}

\ams{60J27}{60J35}

\par
\vspace{5mm}
 \setcounter{section}{1}
 \setcounter{equation}{0}
 \setcounter{theorem}{0}
 \setcounter{lemma}{0}
 \setcounter{corollary}{0}
\noindent {\large \bf 1. Introduction}
\vspace{3mm}
\par
 Markov branching processes occupy a major niche in the
   theory and applications of probability theory. Good references
   are, among many others, Harris~\cite{1963-Harris-p}, Athreya and Ney~\cite{1972-Athreya and Ney-p} and
   Asmussen and Hering~\cite{Asm83}, Athreya and Jagers~\cite{Ath96}. Within this framework both
   state-independent and state-dependent immigration have
   important roles to play. For the former, Sevast'yanov~\cite{Seva57} and Vatutin~\cite{Vat74}-\cite{Vat77}
   considered a branching process with state-independent
   immigration. Aksland~\cite{Ask75}
   considered a modified birth-death process where the
   state-independent immigration is imposed on a simple
   birth-death underlying structure.
   On the other hand, the latter (state-dependent
   immigration) can be traced to Foster~\cite{Fost71} and Pakes~\cite{PAK71} who
   considered a discrete branching process with immigration
   occurring only when the process occupies state $0$.
   Yamazato~\cite{Yama75} investigated the continuous-time version,
   See also the discussion in Pakes and Tavar\'{e}~\cite{PT81}.
\par
The decay parameter and the quasi-stationary distributions are
closely linked with the development of continuous time Markov
chains. The idea of using quasi-stationary distribution can be
traced back at least to the early work of
Yaglom~\cite{1947-Yaglom-p795-798}, who considered the long-run
behavior, in a sense of the subcritical Galton-Watson process. The
decay parameter was developed by Kingman in early 1960's. Beginning
with the pioneering and remarkable work of
Kingman~\cite{1963-Kingman-p337-358} and
Vere-Jones~\cite{1962-Vere-Jones-p7-28}, this extremely useful
theory has been flourished owing to many important researches,
including the significant contributions made by
Flaspohler~\cite{1974-Flas-p351-356}, Pakes~\cite{PT81},
Pollett~\cite{1988-Pollett-p600-621}--\cite{1999-Pollett-p268-272},
Darroch and Seneta~\cite{1967-Darroch and Seneta-p192-196},
Kelly~\cite{1983-Kelly-p143-160},
Kijima~\cite{1963-Kijima-p509-517}, Nair and Pollett~\cite{1993-Nair
and Pollett-p82-102}, Tweedie~\cite{1974-Tweedie-p485-493}, Van
Doorn~\cite{1985-Van Doorn-p514-530} and many others.
\par
$n$-type Markov branching process has been discussed in
Harris~\cite{1963-Harris-p}, Athreya and Ney~\cite{1972-Athreya and
Ney-p}. The aim of this paper is to consider the $n$-type branching
processes with immigration and resurrection, which is the further
extension of the $n$-type Markov branching process. We will discuss
the extinction behavior, recurrence property and decay property. The
evolution of a $n$-type branching process with immigration can be
intuitively described as follows:
\par
(i)\ Consider a system involving $n$ types of particles. The life
length of a type $i$ particle is exponentially distributed with mean
$\theta_{i},i=1,\cdots,n$.
\par
(ii)\ Particles give ``offspring" independently. When a type $i$
particle splits(dies), it produces $j_1$ particles of type 1,
$\cdots$, $j_n$ particles of type $n$, with probability
$p_{j_1,\cdots,j_n}^{(i)}$.
\par
(iii)\ When the system is empty, the immigration still occurs.
\par
(iv)\ If particles are migrant from the external environment, then
they will follow the same reproductive rules as the original
particles in the system .
\par
We begin our research by giving the formal definition of $n$-type
branching process with immigration. Throughout this paper, we adopt
the following conventions:
\par
(C-1)\ ${\bf Z}^n_+=\{(i_1,\cdots,i_n): i_1,\cdots,i_n\in {\bf
Z}_+\}$, for any ${\bf i}=(i_1,\cdots,i_n)\in {\bf Z}^n_+$, denote
$|{\bf i}|=\sum_{k=1}^n i_k$.
\par
(C-2)\ $[0,1]^n=\{(u_1,\cdots,u_n):0\leq u_1,\cdots,u_n\leq 1\}$.
\par
(C-3)\ $\chi_{_{{\bf Z}_+^n}}(\cdot)$ is the indicator of ${\bf
Z}_+^n$.
\par
(C-4)\ ${\bf 0}=(0,\cdots,0)$, ${\bf 1}=(1,\cdots,1)$, ${\bf
e}_i=(0,\cdots,1_i,\cdots,0)$ are vectors in $[0,1]^n$.

\begin{definition}
\label{def1.1}
   A $q$-matrix $Q=(q_{{\bf i j}};{\bf i},{\bf j}\in {\bf Z}^n_+)$ is called an $n$-type
   branching with immigration
   $q$-matrix (henceforth referred to as a $n$TBI $q$-matrix) if it takes the following form:
$$
q_{{\bf i j}}
   =\begin{cases}
   h_{\bf j}\cdot\chi_{_{{\bf Z}_+^n}}({\bf j}),
    &  \mbox{if \ $|{\bf i}|=0$}\\
   \sum_{k=1}^ni_kb^{(k)}_{{\bf j}-{\bf i}+{\bf e}_k}\cdot \chi_{_{{\bf Z}_+^n}}({\bf j}-{\bf i}+{\bf e}_k)
   +a_{\bf j-i}\cdot\chi_{_{{\bf Z}_+^n}}({\bf j}-{\bf i}),
    &  \mbox{if \ $|{\bf i}|>0$}\\
     0,              & \mbox{otherwise}
    \end{cases}
  \eqno(1.1)
$$
 where
$$
\begin{cases}
 h_{\bf j}\geq 0({\bf j}\neq {\bf 0}), 0<\sum_{{\bf j}\neq {\bf 0}}h_{\bf
 j}=-h_{\bf 0}<\infty; \\
 a_{\bf j}\geq 0({\bf j}\neq {\bf 0}),0<\sum_{{\bf j}\neq {\bf 0}}a_{\bf
 j}=-a_{\bf 0}<\infty \\
 b^{(k)}_{{\bf j}}\geq 0 \ ({\bf j}\neq {\bf e}_k),\  0<\sum_{{\bf j}\neq {\bf e}_k}b^{(k)}_{\bf
 j}=
 -b^{(k)}_{{\bf e}_k}<\infty, \ \ k=1,\cdots,n.
 \end{cases}\eqno(1.2)
$$
\end{definition}
\begin{remark}
 $\{h_{\bf j};{\bf j}\neq {\bf 0}\}$ denotes the ``resurrection rate",
 $\{a_{\bf j};{\bf j}\neq {\bf 0}\}$ denotes the
``immigration rate" whilst $\{b_{\bf j};{\bf j}\neq {\bf e}_k\}$
denotes the ``branching rate".
\end{remark}
\begin{definition}
\label{def1.2}  An $n$-type branching process with
immigration(henceforth referred to simply as a $n$TBIP) is a
continuous-time Markov chain with state space ${\bf Z}^n_+$, whose
transition function $P(t)=(p_{{\bf ij}}(t);$ ${\bf i,j}\in {\bf
Z}^n_+)$ satisfies Kolmogorov forward equation
$$
 P'(t)=P(t)Q  \eqno(1.3)
$$
where $Q$ is a $n$TBI $q$-matrix as given in $(1.1)-(1.2)$.
\end{definition}
\par
Here we have defined the $Q$-process as the corresponding transition
$P(t)$ rather than the process itself. In fact, for convenience, we
shall freely use this term to denote either of them in this paper.
This is, of course, commonly accepted and will not cause any
confusion.
\par
By Kingman~\cite{1963-Kingman-p337-358}, we know that there exists a
number $\lambda_C\geq 0$, called the decay parameter of the process
$P(t)$, such that for all ${\bf i,j}\in C$ (where $C$ is a
irreducible class),
\begin{eqnarray*}
  \frac{1}{t}\log p_{{\bf ij}}(t)\rightarrow -{\lambda}_C\ \ \mbox{as}\
  t\rightarrow +\infty.
\end{eqnarray*}
On the other hand, let
\begin{eqnarray*}
  {\mu}_{{\bf ij}}=\inf \{\lambda \geq 0: \int_0^{\infty}e^{\lambda
  t}p_{{\bf ij}}(t)dt=\infty\}=\sup\{\lambda \geq 0: \int_0^{\infty}e^{\lambda
  t}p_{{\bf ij}}(t)dt<\infty\}.
\end{eqnarray*}
By the irreducibility argument, it is fairly easy to show that
${\mu}_{{\bf ij}}$ does not depend on ${\bf i,j}\in C$. Denote the
common value of ${\mu}_{{\bf ij}}$ by $\mu$. It is straightforward
to show that the common abscissa of convergence of these integrals
is just the decay parameter, i.e., $ {\lambda}_C=\mu$.
\par
It is well known that the decay parameter and quasi-stationary
distributions are closely linked with the so-called
$\mu$-subinvariant/invariant measures and
$\mu$-subinvariant/invariant vectors. An elementary but detailed
discussion of this theory can be seen in
Anderson~\cite{-Anderson-p}. For convenience, we briefly repeat
these definitions, tailored for our special models, as follows:
\par
\begin{definition}
Let $Q=(q_{{\bf ij}};{\bf i,j}\in {\bf Z}_+^n)$ be an $n${\rm{TBI}}
$q$-matrix and $C$ be a irreducible class. Assume that $\mu\geq 0$.
A set $(m_{\bf i};{\bf i}\in C)$ of strictly positive numbers is
called a $\mu$-subinvariant measure for $Q$ on $C$ if
$$
 \sum_{{\bf i}\in C}m_{\bf i}q_{{\bf ij}}\leq -\mu m_{\bf
j},\ \ \ \ {\bf j}\in C. \eqno(1.4)
$$
If the equality holds in $(1.4)$, then $(m_{\bf i};{\bf i}\in C)$ is
called a $\mu$-invariant measure for $Q$ on $C$.
\end{definition}
\par
\begin{definition}
Let $P(t)=(p_{{\bf ij}}(t);{\bf i,j}\in {\bf Z}_+^n)$ be an
$n${\rm{TBIP}} and $C$ be a irreducible class. Assume that $\mu\geq
0$. A set $(m_{\bf i};{\bf i}\in C)$ of strictly positive numbers is
called a $\mu$-subinvariant measure for $P(t)$ on $C$ if
$$
  \sum_{{\bf i}\in C}m_{\bf i}p_{{\bf ij}}(t)\leq e^{-\mu t} m_{\bf j},\ \ \ \ {\bf j}\in C.
\eqno(1.5)
$$
If the equality holds in $(1.5)$, then $(m_{\bf i};{\bf i}\in C)$ is
called a $\mu$-invariant measure for $P(t)$ on $C$.
\end{definition}
\par
 The subinvariant/invariant vectors can be similarly defined.
\par
\begin{definition}
Let $P(t)=(p_{{\bf ij}}(t);{\bf i,j}\in {\bf Z}_+^n)$ be an
$n${\rm{TBIP}} and $C$ be a communicating class. Assume that
$(m_{\bf i};{\bf i}\in C)$ is a probability distribution over $C$.
Let $p_{\bf j}(t)=\sum_{{\bf i}\in C}m_{\bf i}p_{{\bf ij}}(t)$, for
${\bf j}\in C, t\geq 0$. If
$$
  \frac{p_{\bf j}(t)}{\sum_{{\bf i}\in C}p_{{\bf i}}(t)}= m_{\bf j},\ \ \ \ {\bf j}\in C, t>0,
\eqno(1.6)
$$
then $(m_{\bf i};{\bf i}\in C)$ is called a quasi-stationary
distribution.
\end{definition}
\par
The deep relationship between invariant measures and
quasi-stationary distributions has been revealed by the important
work of Van Doorn~\cite{1985-Van Doorn-p514-530}, and Nair and
Pollett~\cite{1993-Nair and Pollett-p82-102}.
\par
For the one-dimensional Markov branching processes with immigration,
the extinction probability and exact value of decay parameter are
well-known. The basic aim of this paper is to investigate the
extinction behavior, recurrence property and decay property of
$n$-type Markov branching processes with immigration. Different from
the one-dimensional cases, when a particle of one type in the system
splits, the number of particles of different type may change.
Therefore, the method used in the one-dimensional case fails and
some new approaches should be used in the current situation. In this
paper, we find a new method (see, Theorem~\ref{th3.1}) to
investigate the deep properties of the $n$-type Markov branching
processes with immigration. Furthermore, this new method can be
available in discussing related models and also be available in
solving some kind of partial differential equations.
\par
The structure of this paper is organized as follows. Regularity and
uniqueness criteria together with some preliminary results are
firstly establish in Section 2. In Section 3, we are concentrated on
discussing the absorptive $n$TBIP(i.e., without resurrection) for
which the most interesting problem is the extinction probability. In
section 4, we mainly consider the case $h_{\bf 0}\neq 0$ and the
recurrence criteria are given. In the following Section 5 and
Section 6, we discuss the branching property and decay properties.
Note that if $h_{\bf j}=a_{\bf j}$, then the branching property and
the decay properties of the corresponding process will be
well-discussed and understood. For this reason, we shall assume that
$h_{\bf j}=a_{\bf j}$ in Section 5 and Section 6.

\par
\vspace{5mm}
 \setcounter{section}{2}
 \setcounter{definition}{0}
 \setcounter{equation}{0}
 \setcounter{theorem}{0}
 \setcounter{lemma}{0}
 \setcounter{corollary}{0}
\noindent {\large \bf 2. Preliminary and uniqueness}
 \vspace{3mm}
\par
Since the $q$-matrix $Q$ is determined by the sequences $\{h_{\bf
j};{\bf j}\in {\bf Z}^n_+\}$, $\{a_{\bf j};{\bf j}\in {\bf Z}^n_+\}$
and $\{b^{(i)}_{\bf j}; {\bf j}\in {\bf Z}^n_+\}\ (i=1,\cdots,n)$,
we define their generating functions as
\begin{eqnarray*}
\ \ & &
H(u_1,\cdots,u_n)=\sum_{j_1=0}^{\infty}\cdots\sum_{j_n=0}^{\infty}h_{j_1,\cdots,j_n}u_1^{j_1}\cdots
u_n^{j_n}; \\
\ \ & &
A(u_1,\cdots,u_n)=\sum_{j_1=0}^{\infty}\cdots\sum_{j_n=0}^{\infty}a_{j_1,\cdots,j_n}u_1^{j_1}\cdots
u_n^{j_n}; \\
\ \ & &
B_i(u_1,\cdots,u_n)=\sum_{j_1=0}^{\infty}\cdots\sum_{j_n=0}^{\infty}b^{(i)}_{j_1,\cdots,j_n}u_1^{j_1}\cdots
u_n^{j_n},\ \ \ i=1,\cdots,n.
\end{eqnarray*}
For the sake of convenience in writing, here we write
$\{h_{(j_1,\cdots,j_n)};(j_1,\cdots,j_n)\in{\bf Z}_+^n\}$ as
$\{h_{j_1,\cdots,j_n};(j_1,\cdots,j_n)\in{\bf Z}_+^n\}$,
$\{a_{(j_1,\cdots,j_n)};(j_1,\cdots,j_n)\in{\bf Z}_+^n\}$ as
$\{a_{j_1,\cdots,j_n};(j_1,\cdots,j_n)\in{\bf Z}_+^n\}$ and
$\{b^{(i)}_{(j_1,\cdots,j_n)};(j_1,\cdots,j_n)\in{\bf Z}_+^n\}$ as
$\{b^{(i)}_{j_1,\cdots,j_n};(j_1,\cdots,j_n)\in{\bf Z}_+^n\}$,
 these would
not cause any confusion. It is clear that $A(u_1,\cdots,u_n)$ and
each $B_i(u_1,\cdots,u_n)$ are well defined at least on $[0,1]^n$.
\par
In order to discuss the $n$-type Markov branching processes with
immigration, we need some preparations. In this section, we first
investigate the properties of the generating functions
$H(u_1,\cdots,u_n)$, $A(u_1,\cdots,u_n)$ and
$\{B_i(u_1,\cdots,u_n);i=1,\cdots,n\}$. Let
\begin{eqnarray}
\ \ & & H_{j}(u_1,\cdots,u_n)=\frac{\partial
H(u_1,\cdots,u_n)}{\partial u_j},\ \ j=1,\cdots,n.\nonumber\\
\ \ & & A_{j}(u_1,\cdots,u_n)=\frac{\partial
A(u_1,\cdots,u_n)}{\partial u_j},\ \ j=1,\cdots,n.\nonumber\\
\ \ & & B_{ij}(u_1,\cdots,u_n)=\frac{\partial
B_i(u_1,\cdots,u_n)}{\partial u_j},\ \ i,j=1,\cdots,n.\nonumber\\
\ \ & & g_{ij}(u_1,\cdots,u_n)=\delta_{ij}-
\frac{B_{ij}(u_1,\cdots,u_n)}{b^{(i)}_{{\bf e}_i}},\ \
i,j=1,\cdots,n\nonumber.
\end{eqnarray}
where  $u_1,\cdots,u_n\in [0,1]$ and $\delta_{ij}$ is the Dirac
function. The matrices $(B_{ij}(u_1,\cdots,u_n))$ and
$(g_{ij}(u_1,\cdots,u_n))$ are denoted by $B(u_1,\cdots,u_n)$ and
$G(u_1,\cdots,u_n)$, respectively.
\par
\begin{definition}
Generating functions $\{B_i(u_1,\cdots,u_n); 1\leq i\leq n\}$ is
called singular if there exists an $n\times n$ matrix $M$ such that
$$
(B_1(u_1,\cdots,u_n),\cdots,B_n(u_1,\cdots,u_n))^{'}=M\cdot
(u_1,\cdots,u_n)'
$$
 where $(x_1,\cdots,x_n)'$ denotes the transpose
of the vector $(x_1,\cdots,x_n)$.
\end{definition}
\par
\begin{definition}
 A nonnegative $n\times n$ matrix $A=(a_{ij})$ is called positively
 regular if there exists $N>0$, such that $A^N>0$.
\end{definition}
\par
If $\{B_i(u_1,\cdots,u_n); 1\leq i\leq n\}$ is singular, then each
particle has exactly one offspring, and hence the branching process
will be equivalent to an ordinary finite Markov chain. In order to
avoid discussing such trivial cases, we shall assume throughout this
paper that the following conditions are satisfied:
\par
({\bf A}-1).\ $\{B_i(u_1,\cdots,u_n); 1\leq i\leq n\}$ is
nonsingular;
\par
({\bf A}-2).\ $B_{ij}(1,\cdots,1)<+\infty,\ i,j=1,\cdots,n$;
\par
({\bf A}-3).\ $G(1,\cdots,1)$ is positively regular.

\par
\begin{lemma}
\label{le 2.1}
 $A(u_1,\cdots,u_n)<0$ for all $u_1,\cdots,u_n\in[0,1)$ and
$\lim_{u_1\uparrow 1,\cdots,u_n\uparrow
1}A(u_1,\cdots,u_n)=A(1,\cdots,1)= 0$. Similar property holds for
$H(u_1,\cdots,u_n)$.
\end{lemma}
\begin{proof}
All the conclusions are easy to be proved by some simple algebra
operations and thus we omitted here.\hfill$\Box$
\end{proof}
\par
The following Lemma is a direct consequence of Li and
Wang~\cite{2012-LiWang-p875-894}, thus the proof is omitted here.
\par
\begin{lemma}
\label{le 2.2} Suppose $G(1,\cdots,1)$ is positively regular and
$\{B_i(u_1,\cdots,u_n); 1\leq i\leq n\}$ is nonsingular. Then the
equation
\begin{eqnarray}{\label{2.1}}
 \begin{cases} B_1(u_1,\cdots,u_n)=0; \\
               B_2(u_1,\cdots,u_n)=0; \\
              \ \ \ \ \ \ \ \cdots \\
               B_n(u_1,\cdots,u_n)=0.
 \end{cases}
\end{eqnarray}
has at most two solutions in $[0,1]^n$. Let ${\bf
q}=(q_1,\cdots,q_n)$ and $\rho(u_1,\cdots,u_n)$ denote the smallest
nonnegative solution to (\ref{2.1}) and the maximal eigenvalues of
$B(u_1,\cdots,u_n)$, respectively. Then,
\par
{\rm (i)}\  $q_i$ is the extinction probability when the Feller
minimal process starts at state ${\bf e}_i\ (i=1,\cdots,n)$.
Moreover, if $\rho(1,\cdots,1)\leq 0$, then ${\bf q}={\bf 1}$; while
if $\rho(1,\cdots,1)>0$, then ${\bf q}<{\bf 1}$, i.e.,
$q_1,\cdots,q_n<1$.
\par
{\rm (ii)}\  $\rho(q_1,\cdots,q_n)\leq 0$.
\end{lemma}
\par
\begin{lemma}
\label{le 2.3} Let $P(t)=(p_{\bf ij}(t);{\bf i},{\bf j}\in {\bf
Z}_+^n)$ and $\Phi(\lambda)=(\phi_{\bf ij}(\lambda);{\bf i},{\bf
j}\in {\bf Z}_+^n)$ be the Feller minimal $Q$-function and
$Q$-resolvent, respectively, where $Q$ is given in $(1.1)-(1.2)$.
Then for any ${\bf i}\in {\bf Z}_+^n$ and $(u_1,\cdots, u_n)\in
[0,1)^n$, we have
\begin{eqnarray}{\label{2.2}}
\aligned \frac{\partial F_{\bf i}(t,u_1,\cdots,u_n)}{\partial t}=&
H(u_1,\cdots, u_n)p_{\bf i 0}(t)+A(u_1,\cdots,u_n)\sum_{{\bf j}\in
{{\bf Z}_+^n\setminus {\bf
0}}}p_{\bf ij}(t)u_{1}^{j_{1}}\cdots u_{n}^{j_{n}}\\
&+\sum _{k=1}^{n}B_k(u_1,\cdots,u_n)\frac{\partial F_{\bf
i}(t,u_1,\cdots,u_n)}{\partial u_k}
\endaligned
\end{eqnarray}
 where $F_{\bf i}(t,u_1,\cdots,u_n)=\sum_{{\bf j}\in
{\bf Z}_{+}^{n}} p_{\bf i  j}(t)u_{1}^{j_{1}}\cdots u_{n}^{j_{n}}$,
or in resolvent version
\begin{eqnarray}{\label{2.3}}
\aligned &\lambda \Phi_{\bf i}(\lambda,u_1,\cdots,
u_n)-u_1^{i_1}\cdots
u_n^{i_n}\\
=&H(u_1,\cdots,u_n)\phi_{\bf i0}(\lambda)+A(u_1,\cdots,u_n)\sum_{{\bf
j}\in {{\bf Z}_+^n\setminus {\bf
0}}}\phi_{\bf ij}(\lambda)u_{1}^{j_{1}}\cdots
u_{n}^{j_{n}}\\
&+\sum _{k=1}^{n}B_k(u_1,\cdots,u_n)\frac{\partial \Phi_{\bf
i}(\lambda,u_1,\cdots,u_n)}{\partial u_k}
\endaligned
\end{eqnarray}
where $\Phi_{\bf i}(\lambda,u_1,\cdots,u_n)=\sum_{{\bf j}\in {\bf
Z}_{+}^{n}}\phi_{\bf ij}(\lambda)u_{1}^{j_{1}}\cdots u_{n}^{j_{n}}$.
\end{lemma}
\begin{proof}
By the Kolmogorov forward equations, we have that for any ${\bf i},
{\bf j}\in {\bf Z}_{+}^{n}$,
$$
p'_{\bf ij}(t)=\sum_{{\bf k}\in {\bf Z}_+^n}p_{\bf
ik}(t)[\sum_{l=1}^nk_lb^{(l)}_{{\bf j}-{\bf k}+{\bf e}_l}\cdot
\chi_{_{{\bf Z}_+^n}}({\bf j}-{\bf k}+{\bf e}_l)
   +a_{\bf j-k}\cdot\chi_{_{{\bf Z}_+^n}}({\bf j}-{\bf k})(1-\delta_{\bf 0 k})+h_{\bf j}\cdot \delta _{\bf 0 k}]
$$
Multiplying $u_1^{j_1}\cdots u_n^{j_n}$ on both sides of the above
equality and summing over ${\bf Z}_+^n$ we immediately obtain
(\ref{2.2}). Taking Laplace transform on both sides of (\ref{2.2})
immediately yields (\ref{2.3}).\hfill$\Box$
\end{proof}

\par
\begin{lemma}
\label{le 2.4}\ Suppose $G(1,\cdots,1)$ is positively regular and
$\{B_i(u_1,\cdots,u_n); 1\leq i\leq n\}$ is nonsingular. If
$\rho(1,\cdots, 1)\leq 0$, then the $Q$-function is honest.
\begin{proof}
By Lemma 2.5 of Li and Wang~\cite{2012-LiWang-p875-894}, we know
that if $\rho(1,\cdots, 1)\leq 0$, then ${\bf q}={\bf 1}$.
\par
 Denote
$$
 r^*=\sup\{r\geq 0; B_k(u_1,\cdots,u_n)=r,\ k=1,\cdots,n \ \mbox{has a solution in}\
 [0,1]^n\}.
$$
By Lemma 2.7 of Li and Wang~\cite{2012-LiWang-p875-894}, we know
that $r^*>0$ and for any $r\in (0,r^*]$, there exist $u_1(r),\cdots,
u_n(r)\in [0,1)$ such that
$$
B_k(u_1(r),\cdots,u_n(r))=r,\ k=1,\cdots,n
$$
and moreover,
$$
   \lim_{r\downarrow 0}u_k(r)=1,\ \ k=1,\cdots,n.
$$
Letting $u_k=u_k(r),\ (k=1,\cdots, n)$ in (\ref{2.2}) and letting
$r\downarrow 0$ yield
$$
\sum_{{\bf j}\in {\bf Z}_{+}^{n}} p_{\bf i  j}(t)\geq 1
$$
i.e., $\sum_{{\bf j}\in {\bf Z}_{+}^{n}} p_{\bf i  j}(t)= 1$, then
the $Q$-process is honest. \hfill$\Box$
\end{proof}
\end{lemma}
\par
Having completed the preparation, we now prove that for any given
$n$TBI $q$-matrix $Q$ defined in $(1.1)-(1.2)$, there always exists
exactly one $Q$-process satisfying Kolmogorov forward equation.
\par
\begin{theorem}
\label{th 2.1} Let $Q$ be a $n$TBI $q$-matrix defined as {\rm
(1.1)}--{\rm (1.2)}. Then there exists exactly one $n$TBIP, i.e.,
the Feller minimal process.
\begin{proof}
By Lemma~\ref{le 2.4}, We only need to consider the cases that
$\rho(1,\cdots, 1)>0$ or $\sum_{k=1}^nB_k(1,\cdots,1)<0$.
 For this purpose, we will show that the equations
\begin{eqnarray}\label{2.4}
\begin{cases}
{\bf \eta}(\lambda I-Q)=0,\ \ \eta_{\bf j}\geq 0, \ {\bf j}\in{\bf Z}_+^n,\\
 \sum_{{\bf j}\in {\bf Z}_+^n}\eta_{\bf j}<+\infty
 \end{cases}
\end{eqnarray}
have only trivial solution. Suppose that the contrary is true and
let $\eta =(\eta_{\bf j};\ {\bf j}\in {\bf Z}_+^n)$ be a non-trivial
solution of (\ref{2.4}) corresponding to $\lambda=1$. Then, by
(\ref{2.4}) we have
\begin{eqnarray}\label{2.5}
   \eta_{\bf j}=\sum_{{\bf k}\in {\bf Z}_+^n}\eta_{\bf
   k}(\sum_{i=1}^nk_ib^{(i)}_{{\bf j}-{\bf k}+{\bf e}_i}\cdot \chi_{_{{\bf Z}_+^n}}({\bf j}-{\bf k}+{\bf e}_i)
   +a_{{\bf j}-{\bf k}}\cdot \chi_{_{{\bf Z}_+^n}}({\bf j}-{\bf k})(1-\delta_{\bf 0 k})+h_{\bf j}\cdot \delta_{\bf 0 k}).
\end{eqnarray}
Multiplying $u_1^{j_1}\cdots u_n^{j_n}$ on both sides of
$(\ref{2.5})$ and using some algebra yields that
\begin{eqnarray*}
 \eta(u_1,\cdots,u_n)
=&\;& \sum_{i=1}^nB_i(u_1,\cdots,u_n)\cdot
\frac{\partial\eta(u_1,\cdots,u_n)}{\partial
u_i}\\
&\;&  +A(u_1,\cdots,u_n)(\eta(u_1,\cdots,u_n)-\eta_{\bf
0})+H(u_1,\cdots,u_n)\eta_{\bf 0}.
\end{eqnarray*}
i.e.,
\begin{eqnarray}\label{2.6}
 &\;& (1-A(u_1,\cdots,u_n))[\eta(u_1,\cdots,u_n)-\eta_{\bf
 0}]+(1-H(u_1,\cdots,u_n))\eta_{\bf 0} \nonumber \\
&=&\sum_{i=1}^nB_i(u_1,\cdots,u_n)\cdot
\frac{\partial\eta(u_1,\cdots,u_n)}{\partial u_i}.
\end{eqnarray}
 If $\rho(1,\cdots, 1)>0$ or $\sum_{k=1}^nB_k(1,\cdots,1)<0$, then
 by Lemma~\ref{le 2.2} and the irreducibility of
$\tilde{Q}$ we know from that (\ref{2.1}) has a solution
$(q_1,\cdots, q_n)\in (0,1)^n$. Let
$(u_1,\cdots,u_n)=(q_1,\cdots,q_n)$ in (\ref{2.6}), we can see that
the right-hand side of (\ref{2.6}) is zero. Therefore, the left-hand
side of (\ref{2.6}) must be zero, which implies that $\eta_{\bf
j}=0\ (\forall {\bf j} \in {\bf Z}_+^n)$. The proof is
completed.\hfill$\Box$
\end{proof}
\end{theorem}

\par
\vspace{5mm}
 \setcounter{section}{3}
 \setcounter{equation}{0}
 \setcounter{theorem}{0}
 \setcounter{lemma}{0}
 \setcounter{corollary}{0}
 \setcounter{remark}{0}
\noindent {\large \bf 3. Extinction Property }
 \vspace{3mm}
 \par

In this section, we shall discuss the extinction property of the
$n$TBIP in the case that $h_{\bf 0}=0$. In this case, the most
interesting problem is the extinction probability. Let $\tilde{Q}$
denote the corresponding absorptive $n$TBI $q$-matrix and
$\tilde{P}(t)=(\tilde{p}_{\bf ij}(t);{\bf i},{\bf j}\in {\bf
Z}_+^n)$ denote the Feller minimal $\tilde{Q}$-function. Also let
$a_{\bf i0}=\lim_{t\rightarrow \infty}\tilde{p}_{\bf i0}(t)$ be the
extinction probability of $\tilde{P}(t)$ starting at state ${\bf
i}$. In order to discuss the extinction property, we need the
following important result, which plays a key role in our
discussion.
\par
\begin{theorem}
\label{th3.1} Suppose that $G(1,\cdots,1)$ is positively regular,
$\{B_i(u_1,\cdots,u_n);1\leq i\leq n\}$ is nonsingular. If
$B_1(0,\cdots,0)>0$, then the system of equations
\begin{eqnarray} \label{3.1}
\begin{cases}
   u'_k(u)=\frac{B_k(u,u_2,\cdots,u_n)}{B_1(u,u_2,\cdots,u_n)},\ \ & 2\leq k \leq n \\
   u_k|_{u=0}=0,\ \ \ &2\leq k \leq n
\end{cases}
\end{eqnarray}
has a unique solution $(u_k(u);2\leq k \leq n)$. Furthermore, this
solution satisfies
\par
$(i)$\ $(u_k(u);2\leq k \leq n)$ is well defined on $[0,q_1]$;
\par
$(ii)$\ $u'_k(0)\geq 0$ and $u'_k(u)>0$ for all $u\in (0,q_1)$ and
$2\leq k \leq n$;
\par
$(iii)$\ $u_k(q_1)=q_k,\ \ 2\leq k \leq n$.
\end{theorem}
\begin{proof}
 Since $B_1(0,\cdots,0)>0$, we know that $B_1(u,0,\cdots,0)=0$ has
a positive root $u^*\in (0,1]$. For any $\varepsilon>0$,
$\{\frac{B_k(u,u_2,\cdots,u_n)}{B_1(u,u_2,\cdots,u_n)};\ 2\leq k
\leq n\}$ satisfy Lipschitz condition on $[0,u^*-\varepsilon]\times
[0,1]^{n-1}$, therefore, by the theory of differential equations,
(\ref{3.1}) has a unique solution $(u_k(u);2\leq k \leq n)$ defined
on $[0,u^*-\varepsilon]$. Furthermore, (\ref{3.1}) has a unique
solution $(u_k(u);2\leq k \leq n)$ defined on $[0,u^*)$ since
$\varepsilon
>0$ is arbitrary.
\par
We claim that $u'_k(u)\geq 0\ (2\leq k \leq n)$ for all $u\in
[0,u^*)$. In fact, if there exist $u\in [0,u^*)$ and $2\leq k\leq n$
such that $u'_k(u)<0$, denote
$$
\tilde{u}=\inf\{u\in [0,u^*): \  u'_k(u)<0\ \mbox{for some }\ k\in
\{2,\cdots, n\}\}
$$
and
\begin{eqnarray*}
H=\{k\in \{2,\cdots,n\}:\ \exists\ \varepsilon>0\ \mbox{s.t.}\
u'_k(u)<0\ \mbox{for}\ u\in (\tilde{u},\tilde{u}+\varepsilon)\}.
\end{eqnarray*}
It is obvious that $H\neq \emptyset$, say $H=\{\tilde{k},\cdots,n\}$
for convenience. It is easy to see that
\begin{eqnarray*}
B_k(\tilde{u},u_2(\tilde{u}),\cdots,u_{n}(\tilde{u}))=0,\ \ k\in H
\end{eqnarray*}
and there exists $\bar{u}\in (\tilde{u},u_1^*)$ such that
$u_k(\bar{u})\geq u_k(\tilde{u})\ (k<\tilde{k})$,
$u_k(\bar{u})<u_k(\tilde{u})\ (k\in H)$ and
\begin{eqnarray}
\label{3.2} B_k(\bar{u},u_2(\bar{u}),\cdots,u_{n}(\bar{u}))<0,\ \
k\in H.
\end{eqnarray}
 Consider
$$
I=\{B_k(\bar{u},u_2(\bar{u}),\cdots,u_{\tilde{k}-1}(\bar{u}),u_{\tilde{k}},\cdots,u_n);\
k\in H\}.
$$
Obviously,
$$
B_k(\bar{u},u_2(\bar{u}),\cdots,u_{\tilde{k}-1}(\bar{u}),u_{\tilde{k}}(\tilde{u}),\cdots,u_n(\tilde{u}))\geq
0;\ k\in H.
$$
Therefore, the smallest nonnegative zeros of $I$ is in
$\prod_{k=\tilde{k}}^n[u_{k}(\tilde{u}),1]$. Combining with
(\ref{3.2}) we know that $u_k(\bar{u})\geq u_k(\tilde{u})\ (k\in H)$
which contradicts with $u_k(\bar{u})<u_k(\tilde{u})\ (k\in H)$.
\par
We further claim that $u'_k(u)> 0\ (2\leq k \leq n)$ for all $u\in
(0,u^*]$. In fact, suppose that there exists $\hat{u}\in (0,u^*]$
such that
$$
  B_k(\hat{u},u_2(\hat{u}),\cdots, u_n(\hat{u}))=0
$$
for some $k\geq 2$. Denote
$$
   \hat{H}=\{k; B_k(\hat{u},u_2(\hat{u}),\cdots, u_n(\hat{u}))=0 \}
$$
and
$$
   \hat{H}^c=\{1,2,\cdots,n\}\setminus \hat{H}.
$$
It is easy to see that $\hat{H}^c\neq \emptyset$. By the
irreducibility of the set of nonzero states we know that there exist
$k\in \hat{H}, j\in \hat{H}^c$ such that
$$
B_{kj}(\hat{u},u_2(\hat{u}),\cdots, u_n(\hat{u}))>0.
$$
On the other hand,
$$
  \lim_{u\uparrow
  \hat{u}}\frac{B_k(u,u_2(u),\cdots,u_n(u))}{u-\hat{u}}=\sum_{i\in
  \hat{H}^c}B_{ki}(\hat{u},u_2(\hat{u}),\cdots, u_n(\hat{u}))\cdot
  u'_i(\hat{u})>0,
$$
which contradicts with $B_k(u,u_2(u),\cdots,u_n(u))\geq 0$ for all
$u\in [0,u^*]$, where $u'_1(\hat{u})=1$.
\par
Since $B_1(u^*,u_2(u^*),\cdots,u_n(u^*))>B_1(u^*,0,\cdots,0)=0$, we
can apply the mathematics induction to prove that the solution of
(\ref{3.1}) can be uniquely extended to $[0,q_1)$. Now, we claim
that
$$
u_k(q_1)=\lim_{u\uparrow q_1}u_k(u)=q_k,\ \ \ k\geq 2.
$$
Indeed, since $B_k(u,u_2(u),\cdots,u_n(u))>0,\ (k\geq 1)$ for all
$u\in(0,q_1)$, it can be easily seen that $u_k(u)\in (0,q_k)\ (k\geq
2)$ for all $u\in (0,q_1)$ and therefore, $u_k(q_1)\in (0,q_k]$ for
all $k\geq 2$. If $u_k(q_1)<q_k$ for some $k\geq 2$, denote
$$
   M=\{k\geq 2; u_k(q_1)<q_k \},\ \ \ \ \ \ M^c=\{1,2,\cdots,n\}\setminus M.
$$
It follows from the irreducibility of the set of nonzero states we
know that there exists $j\in M^c$ such that
$$
  \lim_{u\uparrow
  q_1}B_j(u,u_2(u),\cdots,u_n(u))=B_j(q_1,u_2(q_1),\cdots,u_n(q_1))<0,
$$
which contradicts with $B_j(u,u_2(u),\cdots,u_n(u))>0$ for all $u\in
(0,q_1)$.
\hfill $\Box$
\end{proof}
\par
\begin{corollary}
\label{cor3.1}\ Suppose that $G(1,\cdots,1)$ is positively regular,
$\{B_i(u_1,\cdots,u_n);1\leq i\leq n\}$ is nonsingular. If
$B_1(0,\cdots,0)>0,\ B_2(0,\cdots,0)>0$, then the system of
equations
\begin{eqnarray} \label{3.3}
\begin{cases}
   u'_k(u)=\frac{B_k(u_1,u,\cdots,u_n)}{B_2(u_1,u,\cdots,u_n)},\ \ &k\neq 2 \\
   u_k|_{u=0}=0,\ \ \ &k\neq 2
\end{cases}
\end{eqnarray}
has the same solution as $(\ref{3.1})$.
\end{corollary}
\par
\begin{proof}
By Theorem~\ref{th3.1}, we know that $(\ref{3.3})$ has a unique
solution. For convenience, we denote the solutions to $(\ref{3.3})$
by $(u_1(u_2),u_3(u_2),\cdots,u_n(u_2))$. Since $u'_1(u_2)>0$ for
all $u_2\in [0,q_2)$, we know that the function $u_1(u_2),\ (u_2\in
[0,q_2))$ has inverse function $u_2=f_2(u_1),\ (u_1\in [0,q_1))$
satisfying $\frac{df_2}{du_1}=1/u'_1$. Let
$u_k=f_k(u_1)=u_k(f_2(u_1))\ (u_1\in [0,q_1])$ for $k\geq 3$. It can
be easily seen that $u_k=f_k(u_1),\ (k\geq 2)$ is the solution to
$(\ref{3.1})$. \hfill $\Box$
\end{proof}
\par
In this paper, we do not consider the trivial case that any particle
will never dye. Therefore, by Theorem~\ref{th3.1} and
Corollary~\ref{cor3.1}, we will always assume that
$B_1(0,\cdots,0)>0$ without loss of generality and let
$(u_2(u),\cdots,u_n(u))\ (u\in [0,q_1])$ denote the unique solution
to $(\ref{3.1})$.
\par
Before stating our main result in this section, we first provide two useful lemmas.

\begin{lemma}\label{le3.1}
Let $(p_{\bf ij}(t);{\bf i},{\bf j}\in {\bf Z}_+^n)$ be the Feller
minimal $Q$-function where $Q$ is an absorptive $n$TBI $q$-matrix.
Then for any ${\bf i}\in {\bf Z}_+^n$,
\begin{eqnarray}\label{3.4}
\int_0^\infty p_{\bf ik}(t)dt<\infty, \ \ {\bf k}\neq {\bf 0}
\end{eqnarray}
and thus
\begin{eqnarray}\label{3.5}
\lim_{t\rightarrow\infty}p_{\bf ik}(t)=0, \ \ {\bf i}\in {\bf Z}_+^n, {\bf k}\neq {\bf 0}.
\end{eqnarray}
Moreover, for any ${\bf i}\in {\bf Z}_+^n\setminus {\bf 0}$ and
$(u_1,u_2,\cdots,u_n)\in [0,1)^n$, we have
\begin{eqnarray}\label{3.6}
\sum_{{\bf k}\neq {\bf 0}}(\int_0^\infty p_{\bf ik}(t)dt)\cdot
u_1^{k_1}u_2^{k_2}\cdots u_n^{k_n}<\infty.
\end{eqnarray}
\end{lemma}
\par
\begin{proof}
It follows from the Kolmogorov forward equations that
\begin{eqnarray*}
p_{\bf i0}(t)=\delta_{\bf i0}+b_{\bf 0}^{(1)}\cdot\int_0^t p_{{\bf
i}{\bf e}_1}(u)du
\end{eqnarray*}
which clearly implies that $\int_0^\infty p_{{\bf i}{\bf
e}_1}(t)dt<\infty$ for all ${\bf i}\in {\bf Z}_+^n$. By repeatedly
using the Kolmogorov forward equations recursively and the
irreducibility of the nonzero states, (~\ref{3.4}) can be easily
proven. Then (~\ref{3.5}) immediately follows from (~\ref{3.4}).
Finally we turn to prove (~\ref{3.6}). For this purpose, we shall
consider two different cases separately.
\par
First, consider the case $0<\rho(1,\cdots,1)\leq \infty$. By
Lemma~\ref{le 2.1}(ii), (\ref{2.1}) has a root
$(q_1,q_2,\cdots,q_n)\in (0,1)^n$. Let
$(\tilde{u}_1,\cdots,\tilde{u}_n)\in \prod_{i=1}^n(q_i,1)$. We claim
that there exists $(\bar{u}_1,\cdots, \bar{u}_n)\in
\prod_{i=1}^n[\tilde{u}_i,1)$ such that
\begin{eqnarray} \label{3.7}
B_i(\bar{u}_1,\cdots,\bar{u}_n)< 0,\ \forall i=1,2,\cdots,n.
\end{eqnarray}
Indeed, let $H_1=\{i;B_i(\tilde{u}_1,\cdots,\tilde{u}_n)> 0\}$. By
Li and Wang~\cite{2012-LiWang-p875-894} we know that $H_1\neq
\{1,2,\cdots,n\}$ since $\rho(1,\cdots,1)>0$. If $H_1=\emptyset$
then $B_i(\tilde{u}_1,\cdots,\tilde{u}_n)\leq 0\ (\forall
i=1,\cdots,n)$. If $H_1\neq \emptyset$ then by Lemma~\ref{le 2.2},
we know that there exists $(u^{(1)}_1,\cdots,u^{(1)}_n)\in
\prod_{i=1}^n[\tilde{u}_i,1)$ such that
$B_i(u^{(1)}_1,\cdots,u^{(1)}_n)=0$ for all $i\in H_1$. Let
$$
  H_2=\{i; B_i(u^{(1)}_1,\cdots,u^{(1)}_n)>0\}
$$
then $H_2\subset \{1,2,\cdots,n\}\setminus H_1$. It is obvious that
$H_1\cup H_2\neq \{1,2,\cdots,n\}$. If $H_2=\emptyset$
 then $B_i(u^{(1)}_1,\cdots,u^{(1)}_n)\leq 0\ (\forall i=1,\cdots,n)$. If
 $H_2\neq\emptyset$ then by Lemma~\ref{le 2.2}, we know that there exists
$(u^{(2)}_1,\cdots,u^{(2)}_n)\in \prod_{i=1}^n[u^{(1)}_i,1)$ such
that $B_i(u^{(2)}_1,\cdots,u^{(2)}_n)=0$ for all $i\in H_1\cup H_2$.
By repeatedly using the same argument and noting $\{1,2,\cdots,n\}$
is a finite set, we can obtain $H_1,\ H_2,\ \cdots, H_m$ such that
$H_{m+1}=\emptyset$ and hence $B_i(u^{(m)}_1,\cdots,u^{(m)}_n)\leq
0\ (\forall i=1,\cdots,n)$. It is obvious that $H_1\cup\cdots \cup
H_{m}\neq \{1,2,\cdots,n\}$, i.e,
$B_i(u^{(m)}_1,\cdots,u^{(m)}_n)<0$ for all $i\in
\{1,\cdots,n\}\setminus H_1\cup\cdots \cup H_{m}$. By the
irreducibility of nonzero states, we can see that (\ref{3.7}) holds
for $\bar{u}_i$ smaller than (if necessary) but closing to
$u^{(m)}_i$.
\par
By (\ref{2.2}) we know that there exists $k\in
\{1,2,\cdots,n\}\setminus H_1\cup\cdots \cup H_{m}$ such that
\begin{eqnarray*}
\frac{\partial F_{\bf i}(t,\bar{u}_1,\cdots,\bar{u}_n)}{\partial
t}&= & A(\bar{u}_1,\cdots,\bar{u}_n)\sum_{{\bf j}\in {{\bf
Z}_+^n\setminus {\bf 0}}}p_{\bf ij}(t)u_{1}^{j_{1}}\cdots
u_{n}^{j_{n}}\\
&\;& +\sum_{k=1}^nB_k(\bar{u}_1,\cdots,\bar{u}_n)\frac{\partial
F_{\bf i}(t,\bar{u}_1,\cdots,\bar{u}_n)}{\partial u_k}
\end{eqnarray*}
which implies (\ref{3.6}).
\par
 Next consider the case that
$\rho(1,1,\cdots,1)\leq 0$. Let $(\tilde{u}_1,\cdots,\tilde{u}_n)\in
(0,1)^n$. By Theorem~\ref{th3.1}, there exists $s\in
(\tilde{u}_1,1)$ such that $(s,u_2(s),\cdots,u_n(s))\in
\prod_{i=1}^n(\tilde{u}_i,1)$ and hence
 by (\ref{2.2}) and Theorem~\ref{th3.1} we have
\begin{eqnarray*}
1\geq A(s,u_2(s),\cdots,u_n(s))G_{\bf
i}(T,s)+B_1(s,u_2(s),\cdots,u_n(s))\cdot \frac{\partial G_{\bf
i}(T,s)}{\partial s}
\end{eqnarray*}
where $G_{\bf i}(T,s)=\sum_{{\bf j}\in {{\bf Z}_+^n\setminus {\bf
0}}}(\int_0^Tp_{\bf ij}(t)dt)s^{j_1}u_2(s)^{j_2}\cdots
u_n(s)^{j_n}$. (\ref{3.6}) can be obtained immediately from the
above inequality. \hfill $\Box$
\end{proof}
\par
For any ${\bf i}\neq {\bf 0}$, define
\begin{eqnarray}\label{3.8}
  G_{\bf i}(u)=\sum_{\bf j\in {\bf Z}_+^n\setminus{\bf
  0}}(\int_0^{\infty}p_{\bf i j}(t)dt)\cdot
  u^{j_1}[u_2(u)]^{j_2}\cdots [u_n(u)]^{j_n}.
\end{eqnarray}
 From Lemma~\ref{le3.1}, $G_{\bf i}(u)$ is well-defined for $u\in
 [0,q_1)$.

\begin{theorem}\label{th3.2}
For any ${\bf i}\neq {\bf 0}$, $a_{\bf i0} = 1$ if and only if
$\rho(1,\cdots,1)\leq 0$ and $J = +\infty$ where
\begin{eqnarray}\label{3.8a}
J:=\int_0^1\frac{1}{B_1(y,u_2(y),\cdots,u_n(y))}\cdot
e^{\int_0^y\frac{A(x,u_2(x),\cdots,u_n(x))}{B_1(x,u_2(x),\cdots,u_n(x))}dx}dy.
\end{eqnarray}
 More specifically,
\par
$(i)$\ If $\rho(1,\cdots,1)\leq 0$ and $J = +\infty$, then $a_{\bf
i0} = 1({\bf i}\neq {\bf 0})$.
\par
$(ii)$\ If $\rho(1,\cdots,1)\leq 0$ and $J < +\infty$, then
\begin{eqnarray}\label{3.9}
a_{\bf i0}=\frac{\int_0^1\frac{y^{i_1}[u_2(y)]^{i_2}\cdots
[u_n(y)]^{i_n}}{B_1(y,u_2(y),\cdots,u_n(y))}\cdot
e^{\int_0^y\frac{A(x,u_2(x),\cdots,u_n(x))}{B_1(x,u_2(x),\cdots,u_n(x))}dx}dy}
{\int_0^1\frac{1}{B_1(y,u_2(y),\cdots,u_n(y))}\cdot
e^{\int_0^y\frac{A(x,u_2(x),\cdots,u_n(x))}{B_1(x,u_2(x),\cdots,u_n(x))}dx}dy}<1
\end{eqnarray}
\par
$(iii)$\ If $0<\rho(1,\cdots,1)\leq +\infty$ and thus the equation $(\ref{2.1})$ possesses a smallest nonnegative root
${\bf q}=(q_1, u_2(q_1),\cdots,u_n(q_1))\in (0,1)^n$, then
\begin{eqnarray*}
a_{\bf i0}=\frac{\int_0^{q_1} \frac{y^{i_1}u_2(y)^{i_2}\cdots
u_n(y)^{i_n}}{B(y,u_2(y),\cdots,u_n(y))}\cdot
e^{\int_0^y\frac{A(x,u_2(x),\cdots,u_n(x))}{B_1(x,u_2(x),\cdots,u_n(x))}dx}dy}{\int_0^{q_1}\frac{1}{B(y,u_2(y),
\cdots,u_n(y))}\cdot
e^{\int_0^y\frac{A(x,u_2(x),\cdots,u_n(x))}{B_1(x,u_2(x),\cdots,u_n(x))}dx}dy}<\prod_{k=1}^nq_k^{i_k}<1,
\ \ {\bf i}\neq {\bf 0}.
\end{eqnarray*}
\end{theorem}

\par
\begin{proof}
Integrating the equality $(2.2)$ with respect to $t\in [0,\infty)$
and using Theorem $\ref{th3.1}$, we have that for any $u\in [0,1)$
and ${\bf i}\neq {\bf 0}$,
\begin{eqnarray}\label{3.10}
& &a_{\bf i0}-u^{i_1}u_2(u)^{i_2}\cdots u_n(u)^{i_n}\nonumber\\
&=&B_1(u,u_2(u),\cdots,u_n(u))\cdot \frac{\partial F_{\bf
i}(u,u_2(u),\cdots,u_n(u))}{\partial
u_k}+A(u,u_2(u),\cdots,u_n(u))\cdot G_{\bf
i}(u)\nonumber\\
&=&B_1(u,u_2(u),\cdots,u_n(u))\cdot G'_{\bf
i}(u)+A(u,u_2(u),\cdots,u_n(u))\cdot G_{\bf i}(u)
\end{eqnarray}
where $G_{\bf i}(u)<+\infty$ is given in $(\ref{3.6})$. First
consider the case $\rho(1,\cdots,1)\leq 0$. Solving the ordinary
differential equation $(\ref{3.10})$ for $u\in [0,1)$ immediately
yields
\begin{eqnarray}\label{3.11}
&\;& G_{\bf i}(u)\cdot
e^{\int_0^u\frac{A(x,u_2(x),\cdots,u_n(x))}{B_1(x,u_2(x),\cdots,u_n(x))}dx}\nonumber\\
&=& \int_0^u\frac{a_{\bf i0}-y^{i_1}u_2(y)^{i_2}\cdots
u_n(y)^{i_n}}{B_1(y,u_2(y),\cdots,u_n(y))}\cdot
e^{\int_0^y\frac{A(x,u_2(x),\cdots,u_n(x))}{B_1(x,u_2(x),\cdots,u_n(x))}dx}dy
\end{eqnarray}
This immediately implies that if $J=+\infty$, then $a_{\bf i0}=1$.
Indeed, if $a_{\bf i0} < 1$, then by letting $s\uparrow 1$ in
$(\ref{3.11})$ we see that the right hand side of $(\ref{3.11})$ tends
to $-\infty$, while the left hand side is always nonnegative, which
is a contradiction. Hence (i) is proven.
\par
Now we turn to (ii). First note that $J<+\infty$ implies
$\int_0^1\frac{A(x,u_2(x),\cdots,u_n(x))}{B_1(x,u_2(x),\cdots,u_n(x))}dx=-\infty$.
Since the left hand side of $(\ref{3.11})$ is clearly nonnegative
and thus so is the right hand side of $(\ref{3.11})$. It follows
that $a_{\bf i0}\geq J^{-1}\cdot
\int_0^1\frac{y^{i_1}u_2(y)^{i_2}\cdots
u_n(y)^{i_n}}{B_1(y,u_2(y),\cdots,u_n(y))}\cdot
e^{\int_0^y\frac{A(x,u_2(x),\cdots,u_n(x))}{B_1(x,u_2(x),\cdots,u_n(x))}dx}dy$.
Therefore, in order to prove (ii), we only need to show that
\begin{eqnarray*}
a_{\bf i0}\leq J^{-1}\cdot \int_0^1\frac{y^{i_1}u_2(y)^{i_2}\cdots
u_n(y)^{i_n}}{B_1(y,u_2(y),\cdots,u_n(y))}\cdot
e^{\int_0^y\frac{A(x,u_2(x),\cdots,u_n(x))}{B_1(x,u_2(x),\cdots,u_n(x))}dx}dy.
\end{eqnarray*}

Take $x^*_{\bf j}=J^{-1}\cdot
\int_0^1\frac{y^{j_1}u_2(y)^{j_2}\cdots
u_n(y)^{j_n}}{B_1(y,u_2(y),\cdots,u_n(y))}\cdot
e^{\int_0^y\frac{A(x,u_2(x),\cdots,u_n(x))}{B_1(x,u_2(x),\cdots,u_n(x))}dx}dy,
{\bf j}\neq {\bf 0}$, then for any ${\bf i}\neq {\bf 0}$,

\begin{eqnarray*}
& &\sum_{\bf j\neq 0}q_{\bf ij}x_{\bf j}^*+q_{\bf
i0}\\
&=&J^{-1}\cdot \int_0^1\frac{\sum_{{\bf j} \in {\bf
Z}_+^n}q_{\bf ij}\cdot y^{j_1}u_2(y)^{j_2}\cdots
u_n(y)^{j_n}}{B_1(y,u_2(y),\cdots,u_n(y))}\cdot
e^{\int_0^y\frac{A(x,u_2(x),\cdots,u_n(x))}{B_1(x,u_2(x),\cdots,u_n(x))}dx}dy\\
&=&J^{-1}\cdot \int_0^1\sum_{k=1}^\infty
i_ky^{i_1}u_2(y)^{i_2}\cdots u_k(y)^{{i_k}-1}u'_k(y) \cdots
u_n(y)^{j_n}\cdot
e^{\int_0^y\frac{A(x,u_2(x),\cdots,u_n(x))}{B_1(x,u_2(x),\cdots,u_n(x))}dx}dy\\
& &+J^{-1}\cdot \int_0^1\frac{y^{j_1}u_2(y)^{j_2}\cdots
u_n(y)^{j_n}A(y,u_2(y),\cdots,u_n(y))}{B_1(y,u_2(y),\cdots,u_n(y))}\cdot
e^{\int_0^y\frac{A(x,u_2(x),\cdots,u_n(x))}{B_1(x,u_2(x),\cdots,u_n(x))}dx}dy\\
&=&0.
\end{eqnarray*}
Here the last equality follows from applying the method of integration by parts. Hence
$(x_{\bf j}^*;{\bf j}\neq {\bf 0})$ is a solution of the equation
\begin{eqnarray*}
\sum_{\bf j\neq 0}q_{\bf ij}x_{\bf j}^*+q_{\bf
i0}=0, \ \ 0\leq x_{\bf j}^*\leq 1, {\bf i}\neq {\bf 0}.
\end{eqnarray*}
By Lemma 3.2 in Li and Chen~\cite{LC2006}, we then have $a_{\bf
i0}\leq x_{\bf i}^*({\bf i}\neq {\bf 0})$ since $a_{\bf i0}$ is the
minimal solution of the above equation. (ii) is proved.
\par
Finally, we consider (iii). Suppose that $0<\rho(1,\cdots,1)\leq
+\infty$. By Lemma 2.1, we know that the equation $(\ref{2.1})$ has
a root $(q_1,u_2(q_1),\cdots,u_n(q_1))\in (0,1)^n$ and $G_{\bf
i}(s)<\infty$ for all $s\in [0,q_1]$. Similarly as in the above, we
only need to show that
\begin{eqnarray*}
a_{\bf i0}&\leq& \lim_{s\uparrow q_1}[\int_0^s\frac{1}{B_1(y,u_2(y),\cdots,u_n(y))}\cdot
e^{\int_0^y\frac{A(x,u_2(x),\cdots,u_n(x))}{B_1(x,u_2(x),\cdots,u_n(x))}dx}dy]^{-1}\cdot\\
&   &\int_0^s\frac{y^{j_1}u_2(y)^{j_2}\cdots
u_n(y)^{j_n}}{B_1(y,u_2(y),\cdots,u_n(y))}\cdot
e^{\int_0^y\frac{A(x,u_2(x),\cdots,u_n(x))}{B_1(x,u_2(x),\cdots,u_n(x))}dx}dy.
\end{eqnarray*}
Since $-a_{\bf 0}>0$, we know by Lemma $\ref{le 2.1}$ that $\int_0^{q_1}\frac{A(x,u_2(x),\cdots,u_n(x))}{B_1(x,u_2(x),
\cdots,u_n(x))}dx=-\infty$ and
$$
\int_0^y\frac{A(x,u_2(x),\cdots,u_n(x))}{B_1(x,u_2(x),\cdots,u_n(x))}dx\leq
\int_0^y\frac{A(q_1,q_2,\cdots,q_n)}{B_1(x,q_2,\cdots,q_n)}dx \leq
C\ln\frac{q_1-y}{q_1}
$$ for $y\in [0,q_1)$, where $C$ is a positive constant.
Hence the integral $\int_0^{q_1}\frac{1}{B_1(y,u_2(y),\cdots,u_n(y))}\cdot
e^{\int_0^y\frac{A(x,u_2(x),\cdots,u_n(x))}{B_1(x,u_2(x),\cdots,u_n(x))}dx}dy$, denoted by $D$, is convergent.
Now by letting
$$
y_{\bf j}^*=D^{-1}\cdot \int_0^{q_1}\frac{1}{B_1(y,u_2(y),\cdots,u_n(y))}\cdot
e^{\int_0^y\frac{A(x,u_2(x),\cdots,u_n(x))}{B_1(x,u_2(x),\cdots,u_n(x))}dx}dy, \ \ {\bf j}\neq {\bf 0},
$$ we may prove similarly as above that $(y_{\bf j}^*;{\bf j}\neq {\bf 0})$ is a solution of the equation
\begin{eqnarray*}
\sum_{\bf j\neq 0}q_{\bf ij}x_{\bf j}^*+q_{\bf
i0}=0, \ \ 0\leq x_{\bf j}^*\leq 1, {\bf i}\neq {\bf 0}.
\end{eqnarray*}
Again by Lemma 3.2 in Li and Chen~\cite{LC2006}, we have $a_{\bf
i0}\leq y_{\bf i}^*({\bf i}\neq {\bf 0})$ which proves the first
equality in $(\ref{3.5})$. The last two assertions in $(\ref{3.5})$
are obvious.\hfill$\Box$
\end{proof}

\par
By Theorem $\ref{th3.2}$, we see that when immigration occurs then
the condition $\rho(1,\cdots,1)\leq 0$(i.e., the death rate is
greater than or equal to the mean birth rate) is no longer
sufficient, though still necessary, for the process to be finally
extinct. A further condition $J=\infty$, which reflects the effect
of immigration, is necessary to guarantee the final extinction.

\par
Having obtained the extinction probability, we are now in a position to consider the extinction time. We shall use
$E_{\bf i}[\tau_0]$ to denote the mean extinction time when the process starts at state ${\bf i}\neq {\bf 0}$.
\par

\begin{theorem}
\label{th3.3} Suppose that $\rho(1,\cdots,1)\leq 0$ and $J=\infty$
where $J$ is given in $(\ref{3.8a})$ and thus the extinction
probability $a_{\bf i0}=1({\bf i}\neq{\bf 0})$. Then for any ${\bf
i}\neq {\bf 0}$, $E_{\bf i}[\tau_0]<\infty$ if and only if
\begin{eqnarray}\label{3.12}
\int_0^1\frac{1-yu_2(y)\cdots u_n(y)-A(y,u_2(y),\cdots,u_n(y))}{B_1(y,u_2(y),\cdots,u_n(y))}dy<\infty
\end{eqnarray}
and in which case, $E_{\bf i}[\tau_0]$ is given by
\begin{eqnarray}\label{3.13}
E_{\bf i}[\tau_0]=\int_0^1\frac{1-y^{i_1}u_2(y)^{i_2}\cdots u_n(y)^{i_n}}{B_1(y,u_2(y),\cdots,u_n(y))}\cdot
e^{-\int_y^1\frac{A(x,u_2(x),\cdots,u_n(x))}{B_1(x,u_2(x),\cdots,u_n(x))}dx}dy.
\end{eqnarray}
\end{theorem}
\par
\begin{proof}
It follows from $(\ref{3.11})$ that
\begin{eqnarray*}
& &\sum_{{\bf j}\neq {\bf 0}}(\int_0^\infty p_{\bf ij}(t)dt)\cdot u^{j_1}u_2(u)^{j_2}\cdots u_n(u)^{j_n}\\
=& &\int_0^u\frac{1-y^{i_1}u_2(y)^{i_2}\cdots
u_n(y)^{i_n}}{B_1(y,u_2(y),\cdots,u_n(y))}\cdot
e^{-\int_y^s\frac{A(x,u_2(x),\cdots,u_n(x))}{B_1(x,u_2(x),\cdots,u_n(x))}dx}dy
\end{eqnarray*}
Letting $s\uparrow 1$, using the honesty condition and applying the Monotone Convergence Theorem then yields
\begin{eqnarray*}
E_{\bf i}[\tau_0]&=&\int_0^\infty(1-p_{\bf i0}(t))dt\\
&=&\sum_{{\bf j}\in {\bf Z}_+^n\setminus {\bf 0}}\int_0^\infty p_{\bf ij}(t)dt\\
&=&\int_0^1\frac{1-y^{i_1}u_2(y)^{i_2}\cdots
u_n(y)^{i_n}}{B_1(y,u_2(y),\cdots,u_n(y))}\cdot
e^{-\int_y^1\frac{A(x,u_2(x),\cdots,u_n(x))}{B_1(x,u_2(x),\cdots,u_n(x))}dx}dy
\end{eqnarray*}
Thus $(\ref{3.13})$ is proved. Finally, it is fairly easy to show
that the expression in $(\ref{3.13})$ is finite if and only if
$(\ref{3.12})$ holds. \hfill $\Box$
\end{proof}

\par
\vspace{5mm}
 \setcounter{section}{4}
 \setcounter{equation}{0}
 \setcounter{theorem}{0}
 \setcounter{lemma}{0}
 \setcounter{corollary}{0}
\noindent {\large \bf 4. Recurrence Property }
 \vspace{3mm}
 \par

 Unlike the previous section, in this section we shall always assume that $h_{\bf 0}<0$ and thus ${\bf 0}$ is no
 longer an absorbing state. For this case, the most important problem is the recurrence property.
 We shall assume that the $n$TBI $q$-matrix $Q$ is regular and thus the $n$TBIP is honest.
\par

\begin{theorem}\label{th4.1}
\par
The $n$TBIP is recurrent if and only if $\rho(1,\cdots,1)\leq 0$ and
$J=+\infty$, where $J$ is given in $(\ref{3.8a})$.
\end{theorem}
\par
\begin{proof} We first prove the ``if" part. By Lemma 4.46 of Chen~\cite{C92}, it is sufficient to prove that the
minimal solution of the equation
\begin{eqnarray}\label{4.1}
x_{\bf i}=\sum_{\bf j\neq \bf 0}\tilde{\pi}_{\bf ij}x_{\bf
j}+\tilde{\pi}_{\bf i0},\ \ {\bf i}\in {\bf Z}_+^n,
\end{eqnarray}
equals $1$ identically, where $(\tilde{\pi}_{\bf ij};{\bf i},{\bf
j}\in {\bf Z}_+^n)$ denote the transition probability of the
embedding chain of the $n$TBIP. Denote
$$
 \pi_{\bf ij}=\begin{cases}\delta_{\bf 0j},&  \mbox{if \ ${\bf i}={\bf 0},{\bf j}\in {\bf Z}_+^n$}\\
 \tilde{\pi}_{\bf ij},&  \mbox{if \ ${\bf i}\neq{\bf 0},{\bf j}\in {\bf Z}_+^n$}.
\end{cases}
$$
If $(x_{\bf i}^*;{\bf i}\in {\bf Z}_+^n)$ is the minimal solution of
the (\ref{4.1}), then it is easy to see that $(x_{\bf i}^*;{\bf
i}\neq{\bf 0})$ is a solution of the equation
$$
x_{\bf i}=\sum_{\bf j\neq \bf 0}\pi_{\bf ij}x_{\bf j}+{\pi}_{\bf
i0},\ \ 0\leq x_{\bf i}\leq1,{\bf i}\neq{\bf 0}.
$$
Indeed, by Lemma 3.2 of Li and Chen~\cite{LC2006} and
Theorem~\ref{th3.2} we immediately see that $x_{\bf i}^*=a_{\bf
i0}=1({\bf i}\neq {\bf 0})$ and hence $x_{\bf i}^*=1({\bf i}\in {\bf
Z}_+^n)$. Therefore, by the Anderson$^{[15]}$, we know that the
$n$TBIP is recurrent.
\par
We now prove the ``only if" part. Assume that either $\rho(1,\cdots,1)\leq 0$ together $J<+\infty$ or
$0<\rho(1,\cdots,1)\leq +\infty$. We shall prove that the process is transient. To this end, it is sufficient
to show that the equation
\begin{eqnarray*}
\sum_{{\bf j}\in {\bf Z}_+^n}\pi_{\bf ij}x_{\bf j}=x_{\bf i}, \ \ {\bf i}\neq {\bf 0},
\end{eqnarray*}
has a non-constant bounded solution. By the Comparison Lemma, we only need to show that the inequality
\begin{eqnarray}\label{4.2}
\sum_{{\bf j}\in {\bf Z}_+^n}\pi_{\bf ij}x_{\bf j}\geq x_{\bf i}, \ \ {\bf i}\neq {\bf 0},
\end{eqnarray}
has a non-constant bounded solution. Now if $0<\rho(1,\cdots,1)\leq
\infty$, By Lemma~\ref{le 2.2}, we know that the equation
$(\ref{2.1})$ has a root ${\bf q}=(q_1,q_2,\cdots,q_n)\in (0,1)^n$.
Let $x_{\bf i}=1-{\bf q}^{\bf i}$, then $(x_{\bf i}; {\bf i}\in {\bf
Z}_+^n)$ is a non-constant bounded solution of $(\ref{4.2})$.
Indeed, for ${\bf i}\neq {\bf 0}$,
\begin{eqnarray*}
&\;& \sum_{{\bf j}\in {\bf Z}_+^n}\pi_{\bf ij}x_{\bf
j}\\
&=&\frac{1}{-\sum_{k=1}^ni_kb_{{\bf e}_k}^{(k)}-a_{\bf 0}}[\sum_{\bf
j\neq i}\sum_{k=1}^ni_kb_{{\bf j-i}+e_k}^{(k)}(1-{\bf q}^{{\bf j}})
+\sum_{\bf j\neq i}a_{\bf j-i}(1-{\bf q}^{\bf j})]\\
&=&\frac{1}{-\sum_{k=1}^ni_kb_{{\bf e}_k}^{(k)}-a_{\bf
0}}[-\sum_{k=1}^ni_kb_{e_k}^{(k)}-a_{\bf
0}+\sum_{k=1}^ni_kb^{(k)}_{e_k}{\bf
q}^{{\bf i}}-{\bf q}^{\bf i}(A({\bf q})-a_{\bf 0})]\\
&=&1-{\bf q}^{\bf i}+\frac{-{\bf q}^{\bf i}A({\bf q})}
{-\sum\limits_{k=1}^ni_kb_{{\bf e}_k}^{(k)}-a_{\bf 0}}\\
&\geq&1-{\bf q}^{\bf i}=x_{\bf i}.
\end{eqnarray*}

If $\rho(1,\cdots,1)\leq 0$ and $J=+\infty$, then by letting
\begin{eqnarray*}
x_{\bf j}^*=J^{-1}\int_0^1\frac{y^{i_1}u_2(y)^{i_2}\cdots u_n(y)^{i_n}}{B_1(y,u_2(y),\cdots,u_n(y))}\cdot
e^{\int_0^y\frac{A(x,u_2(x),\cdots,u_n(x))}{B_1(x,u_2(x),\cdots,u_n(x))}dx}dy, \ \ {\bf j}\in {\bf Z}_+^n,
\end{eqnarray*}
we may easily verify as in the proof of Theorem~\ref{th3.2} that
$(x_{\bf j}^*; {\bf j}\in {\bf Z}_+^n)$ is a non-constant bounded
solution of $(\ref{4.2})$. \hfill$\Box$
\end{proof}

\par
\begin{theorem}\label{th4.2}
The $n$TBIP is positive recurrent (i.e., ergodic) if and only if
$\rho(1,\cdots,1)\leq 0$ and
\begin{eqnarray}\label{4.3}
\int_0^1\frac{-A(y,u_2(y),\cdots,u_n(y))-H(y,u_2(y),\cdots,u_n(y))}{B_1(y,u_2(y),\cdots,u_n(y))}dy<\infty.
\end{eqnarray}
Moreover, if $\rho(1,\cdots,1)<0$ and
$\sum_{j=1}^n[A_j(1,\cdots,1)+H_j(1,\cdots,1)]<\infty$, then the
process is exponentially ergodic.
\end{theorem}
\par
\begin{proof}
Suppose that $\rho(1,\cdots,1)\leq 0$ and ${(\ref{4.3})}$ holds. By
Chen~\cite{C92}, in order to prove the positive recurrence, we only
need to show that the equation
\begin{eqnarray*}
\begin{cases}
\sum_{{\bf j}\in {\bf Z}_+^n}q_{\bf ij}y_{\bf j}\leq -1,\ \ {\bf i}\neq {\bf 0},\\
\sum_{{\bf j}\neq {\bf 0}}q_{\bf 0j}y_{\bf j}<\infty
\end{cases}
\end{eqnarray*}
has a finite nonnegative solution. By the irreducibility property and the fact that $\rho(1,\cdots,1)\leq 0$,
we may get from $(\ref{4.3})$ that
\begin{eqnarray*}
\int_0^1\frac{1-y^{i_1}u_2(y)^{i_2}\cdots
u_n(y)^{i_n}}{B_1(y,u_2(y),\cdots,u_n(y))}\cdot
e^{\int_0^y\frac{A(x,u_2(x),\cdots,u_n(x))}{B_1(x,u_2(x),\cdots,u_n(x))}dx}dy<\infty,\ \ {\bf j}\in {\bf Z}_+^n.
\end{eqnarray*}
Indeed, since $h_{\bf 0}<0$, it is easy to see that there exists a
positive constant $L$ such that $1-yu_2(y)\cdots u_n(y)\leq L\cdot
H(y,u_2(y),\cdots,u_n(y))$, which implies that
\begin{eqnarray*}
\int_0^1\frac{1-y^{j_1}u_2(y)^{j_2}\cdots
u_n(y)^{j_n}}{B_1(y,u_2(y),\cdots,u_n(y))}dy<\infty
\end{eqnarray*}
for any ${\bf j}\in {\bf Z}_+^n$. Now let
\begin{eqnarray*}
y_{\bf j}=e^{-\int_0^1\frac{A(x,u_2(x),\cdots,u_n(x))}{B_1(x,u_2(x),\cdots,u_n(x))}dx}\cdot\int_0^1\frac{1-y^{j_1}
u_2(y)^{j_2}\cdots
u_n(y)^{j_n}}{B_1(y,u_2(y),\cdots,u_n(y))}\cdot e^{\int_0^y\frac{A(x,u_2(x),\cdots,u_n(x))}{B_1(x,u_2(x),
\cdots,u_n(x))}dx}dy,\ \ {\bf j}\in {\bf Z}_+^n,
\end{eqnarray*}
then $0\leq y_{\bf j}<\infty ({\bf j}\in {\bf Z}_+^n)$ and for any ${\bf i}\neq {\bf 0}$,
\begin{eqnarray*}
& &\sum_{{\bf j}\in {\bf Z}_+^n}q_{\bf ij}y_{\bf j}\\
&=&e^{-\int_0^1\frac{A(x,u_2(x),\cdots,u_n(x))}{B_1(x,u_2(x),\cdots,u_n(x))}dx}\cdot\int_0^1
\frac{-\sum_{{\bf j}\in {\bf Z}_+^n}q_{\bf ij}y^{j_1}u_2(y)^{j_2}\cdots
u_n(y)^{j_n}}{B_1(y,u_2(y),\cdots,u_n(y))}\cdot e^{\int_0^y\frac{A(x,u_2(x),\cdots,u_n(x))}
{B_1(x,u_2(x),\cdots,u_n(x))}dx}dy\\
&=&e^{-\int_0^1\frac{A(x,u_2(x),\cdots,u_n(x))}{B_1(x,u_2(x),\cdots,u_n(x))}dx}\cdot\int_0^1
[\frac{-f(y)A(y,u_2(y),\cdots,u_n(y))}{B_1(y,u_2(y),\cdots,u_n(y))}-f'(y)]\cdot
e^{\int_0^y\frac{A(x,u_2(x),\cdots,u_n(x))}
{B_1(x,u_2(x),\cdots,u_n(x))}dx}dy\\
&=&-1
\end{eqnarray*}
where $f(y)=y^{i_1}u_2(y)^{i_2}\cdots u_n(y)^{i_n}$. As to ${\bf
i}={\bf 0}$, it is easy to see that
\begin{eqnarray*}
\sum_{\bf j\neq 0}q_{\bf 0j}y_{\bf j}\leq e^{-\int_0^1\frac{A(x,u_2(x),\cdots,u_n(x))}{B_1(x,u_2(x),\cdots,u_n(x))}dx}
\cdot \int_0^1\frac{-H(y,u_2(y),\cdots,u_n(y))}{B_1(y,u_2(y),\cdots,u_n(y))}dy<\infty.
\end{eqnarray*}
Therefore the $n$TBIP is positive recurrent.
\par
Conversely, suppose that the process is positive recurrent and thus possesses and equilibrium distribution
$(\pi_{\bf j}; {\bf j}\in {\bf Z}_+^n)$, that is
\begin{eqnarray*}
\lim_{t\rightarrow\infty}p_{\bf ij}(t)=\pi_{\bf j}>0\ \ \mbox{and}\ \ \sum_{{\bf j}\in {\bf Z}_+^n}\pi_{\bf j}=1.
\end{eqnarray*}
Letting $t\rightarrow \infty$ in $(\ref{2.2})$ and using the Dominated Convergence Theorem yields
\begin{eqnarray}\label{4.4}
& &H(s,u_2(s),\cdots,u_n(s))\pi_{\bf 0}+A(s,u_2(s),\cdots,u_n(s))\sum_{\bf j\neq 0}\pi_{\bf j}s^{j_1}u_2(s)^{j_2}
\cdots u_n(s)^{j_n}\nonumber\\
&+&\sum_{k=1}^n B_k(s,u_2(s),\cdots,u_n(s))\sum_{\bf j\neq 0}\pi_{\bf j}j_ks^{j_1}u_2(s)^{j_2}\cdots u_k(s)^{j_k-1}
\cdots u_n(s)^{j_n}=0,
\end{eqnarray}
for $s\in [0,1)$.
\par
Note that (\ref{4.4}) implies that $\rho(1,\cdots,1)\leq 0$. Indeed,
since $H(s,u_2(s),\cdots,u_n(s))< 0$ and $A(s,u_2(s),\cdots,u_n(s))<
0$ for all $s\in [0,1)$, which, by the proof of Theorem
$\ref{th3.1}$, implies that $\rho(1,\cdots,1)\leq 0$. Denote
$\pi(s)=\sum_{{\bf j}\in {\bf Z}_+^n}\pi_{\bf
j}s^{j_1}u_2(s)^{j_2}\cdots u_n(s)^{j_n}$, then $(\ref{4.4})$ can be
written as
\begin{eqnarray*}
&&\sum_{k=1}^nB_k(s,u_2(s),\cdots,u_n(s))\pi'(s)+A(s,u_2(s),\cdots,u_n(s))\pi(s)
+\pi_{\bf 0}[H(s,u_2(s),\cdots,u_n(s))\\
& &-A(s,u_2(s),\cdots,u_n(s))]=0,\ \ s\in [0,1),
\end{eqnarray*}
and hence
\begin{eqnarray}\label{4.5}
\pi(s)=\pi_{\bf 0}[1+\int_0^s\frac{-H(y,u_2(y),\cdots,u_n(y))}{B_1(y,u_2(y),\cdots,u_n(y))}\cdot
e^{-\int_y^s\frac{A(x,u_2(x),\cdots,u_n(x))}{B_1(x,u_2(x),\cdots,u_n(x))}dx}dy],\ \ s\in [0,1)
\end{eqnarray}
Letting $s\uparrow 1$ in $(\ref{4.5})$ yields
\begin{eqnarray*}
\lim_{s\uparrow 1}\frac{\int_0^s\frac{-H(y,u_2(y),\cdots,u_n(y))}{B_1(y,u_2(y),\cdots,u_n(y))}
\cdot e^{\int_0^y\frac{A(x,u_2(x),\cdots,u_n(x))}{B_1(x,u_2(x),\cdots,u_n(x))}dx}dy}{e^{\int_0^s\frac{A(x,u_2(x),
\cdots,u_n(x))}{B_1(x,u_2(x),\cdots,u_n(x))}dx}}<\infty.
\end{eqnarray*}
Since
\begin{eqnarray*}
& &\int_0^s\frac{-H(y,u_2(y),\cdots,u_n(y))}{B_1(y,u_2(y),\cdots,u_n(y))}\cdot e^{\int_0^y\frac{A(x,u_2(x),
\cdots,u_n(x))}{B_1(x,u_2(x),\cdots,u_n(x))}dx}dy\\
&\geq&
\int_0^{s_0}\frac{-H(y,u_2(y),\cdots,u_n(y))}{B_1(y,u_2(y),\cdots,u_n(y))}\cdot e^{\int_0^y\frac{A(x,u_2(x),
\cdots,u_n(x))}{B_1(x,u_2(x),\cdots,u_n(x))}dx}dy\\
&>&0
\end{eqnarray*}
for some $s_0\in (0,1)$ as $s\uparrow 1$, we must have $\int_0^1\frac{-A(x,u_2(x),\cdots,u_n(x))}{B_1(x,u_2(x),
\cdots,u_n(x))}dx<\infty$. Hence
\begin{eqnarray*}
& &\lim_{s\uparrow 1}\int_0^s\frac{-H(y,u_2(y),\cdots,u_n(y))}{B_1(y,u_2(y),\cdots,u_n(y))}\\
&\leq&\lim_{s\uparrow 1}\frac{\int_0^s\frac{-H(y,u_2(y),\cdots,u_n(y))}{B_1(y,u_2(y),\cdots,u_n(y))}\cdot
e^{\int_0^y\frac{A(x,u_2(x),\cdots,u_n(x))}{B_1(x,u_2(x),\cdots,u_n(x))}dx}dy}
{e^{\int_0^s\frac{A(x,u_2(x),\cdots,u_n(x))}{B_1(x,u_2(x),\cdots,u_n(x))}dx}}\\
&<&\infty.
\end{eqnarray*}
\par
Hence $(\ref{4.3})$ holds, which completes the proof of the first
part.
\par
Now suppose that $\rho(1,\cdots,1)<0$ and
$\sum_{j=1}^n[A_j(1,\cdots,1)+H_j(1,\cdots,1)]<\infty$. We prove
that the $n$TBIP is exponentially ergodic. By Corollary 4.49 of
Chen~\cite{C92}, it is sufficient to show that there exist two
 constants $C_1\geq 0$, $C_2>0$ and a finite nonnegative function $(f_{\bf i};{\bf i}\in {\bf Z}_+^n)$
 with $\lim_{{\bf i}\rightarrow \infty}f_{\bf i}=+\infty$ such that
\begin{eqnarray*}
\sum_{{\bf j}\in {\bf Z}_+^n}q_{\bf ij}(f_{\bf j}-f_{\bf i})\leq C_1-C_2f_{\bf i},\ \ {\bf i}\in {\bf Z}_+^n.
\end{eqnarray*}
Since $\rho(1,\cdots,1)$ has a positive eigenvector
$(x_1,\cdots,x_n)$, let
$$
C_1=(\sum_{j=1}^nA_j(1,\cdots,1))\vee
(\sum_{j=1}^nH_j(1,\cdots,1))\cdot \max\{x_1,\cdots,x_n\}>0,
$$
$$
 C_2=-\rho(1,\cdots,1)>0
 $$
  and $f_{\bf i}=\sum_{k=1}^ni_kx_k
\ ({\bf i}\in {\bf Z}_+^n)$.
 Then for any ${\bf i}\in {\bf Z}_+^n$,
\begin{eqnarray*}
&&\sum_{{\bf j}\in {\bf Z}_+^n}q_{\bf ij}(f_{\bf j}-f_{\bf i})\\ &=&
\sum_{{\bf j}\in {\bf Z}_+^n}\sum_{k=1}^ni_kb^{(k)}_{{\bf
j-i}+e_k}(f_{\bf j}-f_{\bf i})+\sum_{{\bf j}\in {\bf
Z}_+^n}(\delta_{\bf i,0}h_{\bf j-i}+(1-\delta_{\bf i,0})a_{\bf
j-i})(f_{\bf j}-f_{\bf i})\\
 &=&
\sum_{k=1}^n\sum_{{\bf j}\in {\bf
Z}_+^n}\sum_{l=1}^ni_kb^{(k)}_{{\bf j-i}+e_k}(j_l-i_l)x_l+\sum_{{\bf
j}\in {\bf Z}_+^n}\sum_{l=1}^n(\delta_{\bf i,0}h_{\bf
j-i}+(1-\delta_{\bf i,0})a_{\bf j-i})(j_l-i_l)x_l\\
&=&
\sum_{k=1}^ni_k\sum_{l=1}^nB_{kl}(1,\cdots,1)x_l+\sum_{k=1}^n(\delta_{\bf
i,0}H_l(1,\cdots,1)+(1-\delta_{\bf i,0})A_l(1,\cdots,1)\\
&\leq & C_1-C_2f_{\bf i}.
\end{eqnarray*}
Thus the proof is complete.\hfill $\Box$
\end{proof}

\par
\begin{theorem}\label{th4.3}
Suppose that the $n$TBIP is positive recurrent. Then its equilibrium
distribution $(\pi_{\bf j}; {\bf j}\in {\bf Z}_+^n)$ is given by
\begin{eqnarray}\label{4.6}
\pi(s)=\pi_{\bf 0}[1+
\int_0^s\frac{-H(y,u_2(y),\cdots,u_n(y))}{B_1(y,u_2(y),\cdots,u_n(y))}\cdot
e^{-\int_y^s\frac{A(x,u_2(x),\cdots,u_n(x))}{B_1(x,u_2(x),\cdots,u_n(x))}dx}dy], \ \ s\in [0,1)
\end{eqnarray}
where $\pi(s)=\sum_{{\bf j}\in {\bf Z}_+^n}\pi_{\bf j}s^{j_1}u_2(s)^{j_2}\cdots u_n(s)^{j_n}$.
\end{theorem}
\par
\begin{proof}
$(\ref{4.6})$ follows directly from the proof of Theorem $\ref{th4.2}$(see $(\ref{4.5})$).\hfill$\Box$
\end{proof}

\par
Finally, we have the following conclusion which follows immediately
from Theorem~\ref{th3.3}.
\begin{theorem}\label{th4.4}
The $n$TBIP is never strongly ergodic.
\end{theorem}

\par
\vspace{5mm}
 \setcounter{section}{5}
 \setcounter{equation}{0}
 \setcounter{theorem}{0}
 \setcounter{lemma}{0}
 \setcounter{corollary}{0}
\noindent {\large \bf 5. Branching Property}
 \vspace{3mm}

\par
In the following two sections, we will consider the branching
property and the decay property. For this purpose, we shall assume
that $h_{\bf j}=a_{\bf j}$, i.e., the $q$-matrix takes the following
form:
\begin{eqnarray}\label{5.1}
q_{{\bf i j}}
   =\begin{cases}
   a_{\bf j}\cdot\chi_{_{{\bf Z}_+^n}}({\bf j}),
    &  \mbox{if \ $|{\bf i}|=0$}\\
   \sum_{k=1}^ni_kb^{(k)}_{{\bf j}-{\bf i}+{\bf e}_k}\cdot \chi_{_{{\bf Z}_+^n}}({\bf j}-{\bf i}+{\bf e}_k)
   +a_{\bf j-i}\cdot\chi_{_{{\bf Z}_+^n}}({\bf j}-{\bf i}),
    &  \mbox{if \ $|{\bf i}|>0$}\\
     0,              & \mbox{otherwise}
    \end{cases}
  \end{eqnarray}
 where
\begin{eqnarray}\label{5.2}
\begin{cases}
 a_{\bf j}\geq 0({\bf j}\neq {\bf 0}), 0<\sum_{{\bf j}\neq {\bf 0}}a_{\bf
 j}\leq-a_{\bf 0}<\infty; \\
 b^{(k)}_{{\bf j}}\geq 0 \ ({\bf j}\neq {\bf e}_k),\ 0<\sum_{{\bf j}\neq {\bf e}_k}b^{(k)}_{\bf j}\leq
 -b^{(k)}_{{\bf e}_k}<\infty, \ \ k=1,\cdots,n.
 \end{cases}
 \end{eqnarray}
 \par
 It is well-known that $n$-type Markov branching process possesses branching property. We now discuss the similar
 property of $n$TBIP, the following theorem reveals that $n$TBIP also possesses the branching property if the resurrection
 is same as the immigration.
\par
\begin{theorem}
\label{th5.1} Let $P(t)=(p_{\bf ij}(t);{\bf i}, {\bf j}\in {\bf
Z}_+^n)$ be a transition function. Then the following statements are
equivalent.
\par
$(i)$\ $P(t)$ is the Feller minimal $Q$-function, where $Q$ takes
the form of $(5.1)-(5.2)$.
\par
$(ii)$\ For any ${\bf i}\in {\bf Z}_+^n$, $t\geq 0$, ${\bf
s}\in[-1,1]^n$, we have
\begin{eqnarray}\label{5.3}
F_{\bf i}(t,{\bf s})=F_{\bf 0}(t,{\bf s})\cdot
\prod_{k=1}^n(\sum_{{\bf j}\in {\bf Z}_+^n}\tilde{p}_{{\bf e}_k{\bf
j}}(t){\bf s}^{\bf j})^{i_k}
\end{eqnarray}
where $F_{\bf i}(t,{\bf s})=\sum_{{\bf j}\in {\bf Z}_{+}^{n}}
p_{{\bf i}{\bf j}}(t){\bf s}^{\bf j}\ ({\bf i}\in {\bf Z}_+^n, {\bf
s}\in[-1,1]^n)$ and $(\tilde{p}_{{\bf e}_i{\bf j}}(t);{\bf j}\in{\bf
Z}_{+}^{n})$ is the Feller minimal $\tilde{Q}$-function and
$\tilde{Q}$ is an $n$-type ordinary branching $q$-matrix $($but may
not be conservative$)$.
\par
$(iii)$\ For any ${\bf i}\in {\bf Z}_+^n$, $t\geq 0$, ${\bf
s}\in[-1,1]^n$, we have
\begin{eqnarray}\label{5.4}
F_{\bf i}(t,{\bf s})=F_{\bf 0}(t,{\bf s})\cdot\prod_{k=1}^n(F_{{\bf
e}_k}(t,{\bf s})/F_{\bf 0}(t,{\bf s}))^{i_k}.
\end{eqnarray}
In particular,
\begin{eqnarray}\label{5.5}
p_{\bf i0}(t)=p_{{\bf 0}{\bf 0}}(t)\cdot \prod_{k=1}^n(p_{{\bf
e}_k{\bf 0}}(t)/p_{\bf 00}(t))^{i_k},\ \ \ |{\bf i}|\geq 1
\end{eqnarray}
\end{theorem}
\par
\begin{proof}
$(ii)\Rightarrow(iii)$ is trivial and thus omitted. Therefore, we
only need to prove $(i)\Rightarrow(ii)$ and $(iii)\Rightarrow(i)$.
Note the fact that $B(1,\cdots,1)\leq0$ always holds no matter $Q$
is conservative or not. We first prove $(i)\Rightarrow(ii)$. If
$(i)$ holds, then $P(t)$ as the Feller minimal $Q$-function,
satisfies the Kolmogorov forward equation $P'(t)=P(t)Q$. We now
prove (\ref{5.3}). Let $\tilde{Q}=(\tilde{q}_{\bf ij};{\bf i},{\bf
j}\in {\bf Z}_+^n)$ be defined as follows:
$$
\tilde{q}_{{\bf i j}}
   =\begin{cases}
   \sum_{k=1}^ni_kb^{(k)}_{{\bf j}-{\bf i}+{\bf e}_k}\cdot \chi_{_{{\bf Z}_+^n}}({\bf j}-{\bf i}+{\bf e}_k),
    &  \mbox{if \ $|{\bf i}|>0$},\\
     0,              & \mbox{otherwise}
    \end{cases}
$$
and $(\tilde{p}_{{\bf e}_i{\bf j}}(t);{\bf j}\in{\bf Z}_{+}^{n})$ is
the Feller minimal $\tilde{Q}$-function. Then by
Athreya~\cite{1972-Athreya and Ney-p}, we have
 $$
 \sum_{{\bf j}\in {\bf
Z}_{+}^{n}}\tilde{ p}_{{\bf i}{\bf j}}(t){\bf s}^{\bf
j}=\prod_{k=1}^n(\sum_{{\bf j}\in {\bf Z}_+^n}\tilde{p}_{{\bf
e}_k{\bf j}}(t){\bf s}^{\bf j})^{i_k}.
$$
\par
Now define $\hat{P}(t)=(\hat{p}_{\bf ij}(t);{\bf i},{\bf j}\in {\bf
Z}_+^n)$ by
$$
\hat{p}_{\bf ij}(t)=\sum_{{\bf k}\leq {\bf j}}p_{\bf
0k}(t)\tilde{p}_{\bf ij-k}(t).
$$
It is easily seen that $\hat{P}(t)$ is a $Q$-function. We now show
that $\hat{P}(t)$ satisfies the Kolmogorov forward equation
$\hat{P}'(t)=\hat{P}(t)Q$. Denote $\hat{F}_{\bf i}(t,{\bf
s})=\sum_{{\bf j}\in {\bf Z}_+^n}\hat{p}_{\bf ij}(t){\bf s}^{\bf
j}$, then
$$
\hat{F}_{\bf i}(t,{\bf s})=F_{\bf 0}(t,{\bf
s})\cdot\prod_{k=1}^n(\sum_{{\bf j}\in {\bf Z}_+^n}\tilde{p}_{{\bf
e}_k{\bf j}}(t){\bf s}^{\bf j})^{i_k},\ \ {\bf i}\in {\bf Z}_+^n.
$$
Now we claim that $\hat{F}_{\bf i}(t,{\bf s})$ satisfies
(\ref{5.3}). Note that
\begin{eqnarray*}
\frac{\partial F_{\bf 0}(t,{\bf s})}{\partial t}=\sum
_{k=1}^{n}B_k({\bf s})\frac{\partial F_{\bf 0}(t,{\bf s})}{\partial
s_k}+A({\bf s})F_{\bf 0}(t,{\bf s})
\end{eqnarray*}
it can be easily seen that $\hat{F}_{\bf i}(t,{\bf s})$ satisfies
\begin{eqnarray*}
\frac{\partial \hat{F}_{\bf i}(t,{\bf s})}{\partial t}=\sum
_{k=1}^{n}B_k({\bf s})\frac{\partial \hat{F}_{\bf i}(t,{\bf
s})}{\partial s_k}+A({\bf s})\hat{F}_{\bf i}(t,{\bf s}),\ \ \ {\bf
i}\neq {\bf 0}
\end{eqnarray*}
which implies that $\hat{P}'(t)=\hat{P}(t)Q$. By Theorem 2.1, we
must have $\hat{P}(t)=P(t)$ and hence (\ref{5.3}) holds. $(ii)$ is
proved.
\par
Next we prove $(iii)\Rightarrow (i)$. First note that (\ref{5.4})
implies that $F_{{\bf e}_k}(t,{\bf s})\leq F_{\bf 0}(t,{\bf s})$ for
all $t>0$ and ${\bf s}\in (0,1)^n$. We now further claim that there
exist $\tilde{t}>0$ and $\tilde{\bf s}\in (0,1)^n$ such that
$$
F_{{\bf e}_k}(\tilde{t},\tilde{\bf s})<F_{\bf
0}(\tilde{t},\tilde{\bf s})
$$
Indeed, suppose the converse is true, then $F_{{\bf e}_k}(t,{\bf
s})=F_{\bf 0}(t,{\bf s})$ for all $t>0$ and ${\bf s}\in (0,1)^n$. It
follows that $F_{{\bf e}_k}(t,{\bf s})=F_{\bf 0}(t,{\bf s})$ holds
even for all $t\geq 0$ and ${\bf s}\in [0,1]^n$ since both $F_{{\bf
e}_k}(t,{\bf s})$ and $F_{\bf 0}(t,{\bf s})$ are continuous
functions of $t\geq 0$ and ${\bf s}\in [0,1]^n$. Hence,
$$
p_{{\bf e}_k{\bf j}}(t)=p_{\bf 0j}(t),\ \ t\geq 0,\ \ {\bf j}\in
{\bf Z}_+^n
$$
which contradicts with the fact that $\lim_{t\downarrow0}p_{\bf
ij}(t)=\delta_{\bf ij}$.
\par
Now, it follows from (\ref{5.4}) and $F_{{\bf
e}_k}(\tilde{t},\tilde{\bf s})<F_{\bf 0}(\tilde{t},\tilde{\bf s})$
that
$$
\lim_{{\bf i}\rightarrow\infty}F_{\bf i}(\tilde{t},\tilde{\bf s})=0
$$
which implies that $\lim_{\bf i\rightarrow\infty}p_{\bf
ij}(\tilde{t})=0$ for all ${\bf j}\in {\bf Z}_+^n$. Therefore
$P(t)=(p_{\bf ij}(t); {\bf i}, {\bf j}\in {\bf Z}_+^n)$ is a
Feller-Reuter-Riley transition function. By Anderson (1991), we know
that the corresponding $q$-matrix $Q=(q_{\bf ij};{\bf i}, {\bf j}\in
{\bf Z}_+^n)$ is stable and furthermore $P(t)$ is the Feller minimal
$Q$-function. Now, we rewrite (\ref{5.4}) as
$$
F_{\bf i}(t,{\bf s})\cdot \prod_{k=1}^nF_{\bf 0}^{i_k}(t,{\bf
s})=F_{\bf 0}(t,{\bf s})\cdot \prod_{k=1}^nF_{{\bf e}_k}^{
i_k}(t,{\bf s})
$$
Denoting $b^{(k)}_{\bf 0}=q_{{\bf e}_k{\bf 0}}(k=1,\cdots,n)$ and
$y_{\bf j}=q_{\bf 0j}({\bf j}\in {\bf Z}_+^n)$. Differentiating the
above equality with respect to $t$ and letting $t=0$ yields that for
any ${\bf i}\neq {\bf 0}$,
\begin{eqnarray*}
&F'_{\bf i}(0,{\bf s})\cdot \prod_{k=1}^nF_{\bf 0}^{i_k}(0,{\bf
s})+F_{\bf i}(0,{\bf s})\cdot\sum_{l=1}^n \frac{\prod_{k=1}^nF_{\bf
0}^{i_k}(0,{\bf s})}{F_{\bf
0}(0,{\bf s})}\cdot i_l\cdot F'_{\bf 0}(0,{\bf s})\\
=&F'_{\bf 0}(0,{\bf s})\cdot\prod_{k=1}^nF_{{\bf e}_k}^{i_k}(0,{\bf
s})+F_{\bf 0}(0,{\bf s})\cdot\sum_{l=1}^n \frac{\prod_{k=1}^nF_{{\bf
e}_k}^{i_k}(0,{\bf s})}{F_{{\bf e}_l}(0,{\bf s})}\cdot i_l\cdot
F'_{{\bf e}_l}(0,{\bf s})
\end{eqnarray*}
where $F'_{\bf i}(t,{\bf s})=\sum_{{\bf j}\in {\bf Z}_{+}^{n}}
p'_{{\bf i}{\bf j}}(t){\bf s}^{\bf j}$. Hence,
$$
\sum_{{\bf j}\in {\bf Z}_+^n}q_{\bf ij}{\bf s}^{\bf
j}=-(\sum_{k=1}^n i_k-1)\cdot \sum_{{\bf j}\in {\bf Z}_+^n}y_{\bf
j}{\bf s}^{\bf j+i}+\sum_{k=1}^n i_k\cdot \sum_{{\bf j}\in {\bf
Z}_+^n}q_{{\bf e}_k{\bf j}}{\bf s}^{{\bf j}+{\bf i}-{\bf e}_k}
$$
Comparing the coefficients of ${\bf s}^{\bf j}$ on both sides of the
above equality yields
$$
q_{\bf ij}=\begin{cases} y_{\bf j}\cdot\chi_{_{{\bf Z}_+^n}}({\bf
j}),
    &  \mbox{if \ $|{\bf i}|=0$}\\
   \sum_{k=1}^ni_k(q_{{\bf e}_k{{\bf j}-{\bf i}+{\bf e}_k}}-y_{\bf j-i})\cdot \chi_{_{{\bf Z}_+^n}}({\bf j}-{\bf i}
   +{\bf e}_k)
   +y_{\bf j-i}\cdot\chi_{_{{\bf Z}_+^n}}({\bf j}-{\bf i}),
    &  \mbox{if \ $|{\bf i}|>0$}\\
     0,              & \mbox{otherwise}
\end{cases}
$$
Noting the fact $q_{\bf ij}\geq 0\ (\sum_{k=1}^n i_k>0, {\bf j}>
{\bf i})$ and $q_{\bf ii}\leq 0$ we can see that
$$
b_{{\bf e}_k}^{(k)}:=q_{{\bf e}_k{\bf e}_k}-y_{\bf 0}\leq 0,\ \
b_{{\bf j}+{\bf e}_k}^{(k)}:=q_{{\bf e}_k{\bf j}+{\bf e}_k}-y_{\bf
j}\geq 0, \ \ {\bf j}> {\bf 0}.
$$
Hence,
$$
q_{\bf ij}=\begin{cases} y_{\bf j}\cdot\chi_{_{{\bf Z}_+^n}}({\bf
j}),
    &  \mbox{if \ $|{\bf i}|=0$}\\
   \sum_{k=1}^ni_k b_{{\bf j}-{\bf i}+{\bf e}_k}^{(k)}\cdot \chi_{_{{\bf Z}_+^n}}({\bf j}-{\bf i}+{\bf e}_k)
   +y_{\bf j-i}\cdot\chi_{_{{\bf Z}_+^n}}({\bf j}-{\bf i}),
    &  \mbox{if \ $|{\bf i}|>0$}\\
     0,              & \mbox{otherwise}
\end{cases}
$$
Comparing this with $(5.1)$ implies that $Q$ takes the form of
$(5.1)$ with $a_{\bf j}=y_{\bf j}({\bf j}\in {\bf Z}_+^n)$. Finally,
general theory of continuous-time Markov chain yields
$$
\sum_{{\bf j}\neq {\bf 0}}y_{\bf j}\leq-y_{\bf 0}<+\infty,\ \
\sum_{{\bf j}\in {\bf Z}_+^n}(y_{\bf j}+\sum_{k=1}^ni_k b_{{\bf
j}-{\bf i}+{\bf e}_k}^{(k)})=\sum_{{\bf j}\in {\bf Z}_+^n}q_{\bf
ij}\leq 0
$$
and hence
$$
\sum_{{\bf j}\neq {\bf e}_k}b_{\bf j}^{(k)}\leq-b_{{\bf
e}_k}^{(k)}<+\infty.
$$
Thus $Q$ takes the form of $(4.1)-(4.2)$ with $a_{\bf j}\equiv
y_{\bf j}({\bf j}\in {\bf Z}_+^n)$(but may not be conservative).
\hfill$\Box$
\end{proof}

\par
\vspace{5mm}
 \setcounter{section}{6}
 \setcounter{equation}{0}
 \setcounter{theorem}{0}
 \setcounter{lemma}{0}
 \setcounter{corollary}{0}
 \setcounter{remark}{0}
\noindent {\large \bf 6. Decay Property}
 \vspace{3mm}
\par
In the previous section, we considered branching property in the
case that $h_{\bf j}=a_{\bf j}$. We now discuss the decay parameter
$\lambda_Z$ and related property in such case.
\par
\begin{theorem}
\label{th6.1} Suppose that $G(1,\cdots,1)$ is positively regular,
$\{B_i(u_1,\cdots,u_n); i=1,\cdots,n\}$ is nonsingular and ${\bf
Z}_+^n$ is a communicating class. Then
\begin{eqnarray*}
\lambda_Z\geq -A(q_1,\cdots,q_n),
\end{eqnarray*}
where $(q_1,\cdots,q_n)$ is the minimal nonnegative solution of
(\ref{2.1}) given in Lemma~\ref{le 2.2}.
\begin{proof}
In order to prove $\lambda_Z\geq -A(q_1,\cdots,q_n)$, it follows
from Proposition 5.4.1 in Anderson~\cite{-Anderson-p}, we only need
to show that there exists a $-A(q_1,\cdots,q_n)$-subinvariant vector
for $Q$ on ${\bf Z}_+^n$. In other words, we only need to show that
there exists a positive $(y_{\bf j};{\bf j}\in {\bf Z}_+^n)$ such
that
\begin{eqnarray*}
\sum_{{\bf k}\in {\bf Z}_+^n\setminus{\bf 0}}q_{\bf jk}\ y_{\bf
k}\leq A(q_1,\cdots,q_n)y_{\bf j},\ \ \ {\bf j} \in {\bf Z}_+^n.
\end{eqnarray*}
By Lemma~\ref{le 2.2}, we know that equation (\ref{2.1}) has a
smallest nonnegative solution ${\bf q}=(q_1,\cdots,q_n)\in [0,1]^n$.
Note that ${\bf Z}_+^n$ is a communicating class, we further have
${\bf q}\in (0,1]^n$.
\par

Define
\begin{eqnarray*}
 y_{k}=q_1^{k_1}\cdots q_n^{k_n},\ \ \
 {\bf k}=(k_1,\cdots,k_n)\in {\bf Z}_+^n.
\end{eqnarray*}
Then $y_{\bf k}>0,\ \forall {\bf k}\in {\bf Z}_+^n$. Moreover,
 \begin{eqnarray*}
 \sum_{{\bf k}\in {\bf Z}_+^n}q_{\bf 0
k}\ y_{\bf k}=\sum_{{\bf k}\in {\bf Z}_+^n}a_{\bf k}q_1^{k_1}\cdots
q_n^{k_n}=A(q_1,\cdots,q_n)=A(q_1,\cdots,q_n)y_{\bf 0}.
\end{eqnarray*}
For ${\bf j}\in {\bf Z}_+^n\setminus {\bf 0},i=1,\cdots,n$,
 \begin{eqnarray*}
\sum_{{\bf k}\in {\bf Z}_+^n}q_{\bf jk} y_{\bf k} &=&\sum_{{\bf
k}\in {\bf Z}_+^n}(\sum\limits_{i=1}^nj_i b_{{\bf k}-{\bf j}+{\bf
e}_i}^{(i)})q_1^{k_1}\cdots q_n^{k_n}+\sum_{{\bf
k}\geq {\bf j}}a_{\bf k-j}q_1^{k_1}\cdots q_n^{k_n}\nonumber \\
&=& \sum\limits_{i=1}^n B_i(q_1,\cdots,q_n)j_iq_1^{j_1}\cdots
q_i^{j_i-1}\cdots q_n^{k_n}
+A(q_1,\cdots,q_n)q_1^{j_1}\cdots q_n^{j_n} \nonumber \\
&=& A(q_1,\cdots,q_n)y_{\bf j}.
\end{eqnarray*}
 Which implies that $(y_{\bf j};{\bf j}\in {\bf Z}_+^n)$ is a
$-A(q_1,\cdots,q_n)$-invariant vector for  $Q$ on ${\bf Z}_+^n$.
Therefore, $\lambda_Z \geq -A(q_1,\cdots,q_n)$. \hfill$\Box$
\end{proof}
\end{theorem}
\par
Theorem~\ref{th6.1} gives a low-bound of the decay parameter. The
following theorem further presents the exact value of the decay
parameter.
\par
\begin{theorem}
\label{th6.2} Suppose $G(1,\cdots,1)$ is positively regular,
$\{B_i(u_1,\cdots,u_n); i=1,\cdots,n\}$ is nonsingular and ${\bf
Z}_+^n$ is a communicating class. Then
\begin{eqnarray*}
\lambda_Z= -A(q_1,\cdots,q_n),
\end{eqnarray*}
where $(q_1,\cdots,q_n)$ is the minimal nonnegative solution of
(\ref{2.1}) given in Lemma~\ref{le 2.2}.
\end{theorem}
\begin{proof}
By Theorem~\ref{th6.1}, we only need to prove $\lambda_Z\leq
-A(q_1,\cdots,q_n)$. Similar to the proof of Theorem~\ref{th 2.1},
we still have $q_1,\cdots,q_n>0$. It follows from the Kolmogorov
forward equation that
\begin{eqnarray*}
\aligned
 & \frac{\partial}{\partial t}F_{{\bf
j}}(u_1,\cdots,u_n,t)\\
=&\sum_{i=1}^nB_i(u_1,\cdots,u_n)\frac{\partial}{\partial
u_i}F_{{\bf j}}(u_1,\cdots,u_n,t)+A(u_1,\cdots,u_n)F_{\bf
j}(u_1,\cdots,u_n,t)
\endaligned
\end{eqnarray*}
where $F_{\bf j}(u_1,\cdots,u_n,t)=\sum_{{\bf k}\in {\bf
Z}_+^n}p_{{\bf j}{\bf k}}(t)u_1^{k_1}\cdots
u_n^{k_n},u_1,\cdots,u_n\in(-1,1)$.
\par
 If $(q_1,\cdots,q_n)<{\bf
1}$. Define
 \begin{eqnarray}{\label{6.1}}
 \hat{p}_{\bf ij}(t)=e^{-A(q_1,\cdots,q_n)t}p_{\bf
 ij}(t)\frac{q_1^{j_1}\cdots q_n^{j_n}}{q_1^{i_1}\cdots q_n^{i_n}},\
 \ {\bf i},{\bf j}\in{\bf Z}_+^n, t\geq 0.
 \end{eqnarray}
 Then by Pollett~\cite{1986-Pollett-p203-221}, we know that $\hat{P}(t)=(\hat{p}_{\bf
 ij}(t);{\bf i},{\bf j}\in {\bf Z}_+^n)$ is a
 stationary and honest transition function on ${\bf Z}_+^n$.
 Moreover, it is easy to see that its $q$-matrix $\hat{Q}=(\hat{q}_{\bf ij};{\bf
 i},{\bf j}\in {\bf Z}_+^n)$ is given by
  $$
 \hat{q}_{{\bf i j}}
   =\begin{cases}
   \hat{a}_{\bf j}\cdot\chi_{_{{\bf Z}_+^n}}({\bf j}),
    &  \mbox{if \ $|{\bf i}|=0$}\\
   \sum_{k=1}^ni_k\hat{b}^{(k)}_{{\bf j}-{\bf i}+{\bf e}_k}\cdot \chi_{_{{\bf Z}_+^n}}({\bf j}-{\bf i}+{\bf e}_k)
   +\hat{a}_{\bf j-i}\cdot\chi_{_{{\bf Z}_+^n}}({\bf j}-{\bf i}),
    &  \mbox{if \ $|{\bf i}|>0$}\\
     0,              & \mbox{otherwise}
    \end{cases}
 $$
 where
 $$
 \hat{a}_{\bf j}=a_{\bf j}q_1^{j_1}\cdots
 q_n^{j_n}-A(q_1,\cdots,q_n)\delta_{\bf 0j},\ \ \hat{b}_{\bf
 j}^{(i)}=b_{\bf j}^{(i)} q_1^{j_1}\cdots q_n^{j_n},\ \ ({\bf j}\in
 {\bf Z}_+^n)
 $$
 Obviously, $\hat{Q}$ is a conservative $n$TBI
 $q$-matrix. Let
 \begin{eqnarray*}
 & & \hat{A}(u_1,\cdots,u_n)=\sum_{{\bf j}\in {\bf
 Z}_+^n}\hat{a}_{\bf j}^{(i)}u_1^{j_1}\cdots u_n^{j_n}\\
 & & \hat{B}_i(u_1,\cdots,u_n)=\sum_{{\bf j}\in {\bf
 Z}_+^n}\hat{b}_{\bf j}^{(i)}u_1^{j_1}\cdots u_n^{j_n},\ \
 (i=1,\cdots,n)
 \end{eqnarray*}
 then
 \begin{eqnarray*}
 & &
 \hat{A}(u_1,\cdots,u_n)=A(q_1u_1,\cdots,q_nu_n)-A(q_1,\cdots,q_n)\\
 & & \hat{B}_i(u_1,\cdots,u_n)=B_i(q_1u_1,\cdots,q_nu_n)\ \
 (i=1,\cdots,n)
 \end{eqnarray*}
 and
 $$
 \hat{B}_i(1,\cdots,1)=\hat{A}(1,\cdots,1)=0;\ \ (i=1,\cdots,n).
 $$
 \par
 Moreover, since $(q_1,\cdots,q_n)<{\bf 1}$, we further have
 \begin{eqnarray*}
 & 0<\hat{A}_{j}(1,\cdots,1)=q_jA_{j}(q_1,\cdots,q_n)<+\infty
 \end{eqnarray*}
 and by Theorem~\ref{th 2.1}
 \begin{eqnarray*}
& \rho_{\hat{B}}(1,\cdots,1)\leq 0.
 \end{eqnarray*}
Hence, by Theorem~\ref{th4.1} we know that $\hat{P}(t)$ is
recurrent, i.e.,
$$
\int_0^\infty \hat{p}_{\bf ii}(t)dt=
 \int_0^\infty e^{-A(q_1,\cdots,q_n)t}p_{\bf
 ii}(t)dt=\infty.
 $$
 Therefore, $\lambda_Z\leq
 -A(q_1,\cdots,q_n)$.
\par
  If $(q_1,\cdots,q_n)={\bf 1}$, then for any $\varepsilon>0$,
  define
$$
q_{{\bf i j}}^{(\varepsilon)}
   =\begin{cases}
   a_{\bf j}\cdot\chi_{_{{\bf Z}_+^n}}({\bf j}),
    &  \mbox{if \ $|{\bf i}|=0$}\\
   \sum_{k=1}^ni_kb^{(k)(\varepsilon)}_{{\bf j}-{\bf i}+{\bf e}_k}\cdot \chi_{_{{\bf Z}_+^n}}({\bf j}-{\bf i}+{\bf e}_k)
   +a_{\bf j-i}\cdot\chi_{_{{\bf Z}_+^n}}({\bf j}-{\bf i}),
    &  \mbox{if \ $|{\bf i}|>0$}\\
     0,              & \mbox{otherwise}
    \end{cases}
$$
where $b_{\bf j}^{(k)(\varepsilon)}=b_{\bf j}^{(k)}-\varepsilon
\delta_{{\bf j},{\bf e}_k}$. It is easy to see that
$Q^{(\varepsilon)}=(q_{\bf ij}^{(\varepsilon)};{\bf i},{\bf j}\in
{\bf Z}_+^n)$ is a nonconservative $n$TBI $q$-matrix.
\par
For any $i=1,\cdots,n$, define
\begin{eqnarray*}
B_i^{(\varepsilon)}(u_1,\cdots,u_n) =\sum_{{\bf j}\in{\bf
Z}_+^n}b_{\bf j}^{(i)(\varepsilon)}u_1^{j_1}\cdots u_n^{j_n}
=B_i(u_1,\cdots,u_n)-\varepsilon u_i
\end{eqnarray*}
then we know that the equation
$B_i^{(\varepsilon)}(u_1,\cdots,u_n)=0$ has the minimal nonnegative
solution
$(q_1^{(\varepsilon)},\cdots,q_n^{(\varepsilon)})\in[0,1)^n$.
Moreover, $(q_1^{(\varepsilon)},\cdots,q_n^{(\varepsilon)})\uparrow
(1,\cdots,1)$ as $\varepsilon\downarrow 0$.
\par
Let $ P^{(\varepsilon)}(t)=(p_{\bf ij}^{(\varepsilon)};{\bf i},{\bf
j}\in {\bf Z}_+^n) $ be the Feller minimal
$Q^{(\varepsilon)}$-function, then $p_{\bf
ij}^{(\varepsilon)}(t)\leq p_{\bf ij}(t)$. Indeed, the Feller
minimal $Q^{(\varepsilon)}$-resolvent
$\Phi^{(\varepsilon)}(\lambda)=(\phi^{(\varepsilon)}_{\bf
ij}(\lambda);{\bf i},{\bf j}\in {\bf Z}_+^n)$ is the minimal
nonnegative solution of the Kolmogorov backward equation
$$
\phi_{\bf ij}^{(\varepsilon)}(\lambda)=\frac{\delta_{\bf
ij}}{\lambda+q_{\bf i}^{(\varepsilon)}}+\sum\limits_{{\bf k}\neq
{\bf i}}\frac{q_{\bf ik}^{(\varepsilon)}}{\lambda+q_{\bf
i}^{(\varepsilon)}}\phi_{\bf kj}^{(\varepsilon)}(\lambda),\ \ {\bf
i}\in {\bf Z}_+^n.
$$
Since the Feller minimal $Q$-resolvent $\Phi(\lambda)=(\phi_{\bf
ij}(\lambda);{\bf i},{\bf j}\in {\bf Z}_+^n)$
 is the minimal nonnegative solution
of the Kolmogorov backward equation
$$
\phi_{\bf ij}(\lambda)=\frac{\delta_{\bf ij}}{\lambda+q_{\bf
i}}+\sum\limits_{{\bf k}\neq {\bf i}}\frac{q_{\bf
ik}}{\lambda+q_{\bf i}}\phi_{\bf kj}(\lambda),\ \ {\bf i}\in {\bf
Z}_+^n.
$$
Since $q_{\bf ik}^{(\varepsilon)}= q_{\bf ik},\forall {\bf i}\neq
{\bf k}$ and $q_{\bf i}^{(\varepsilon)}=q_{\bf i}-\sum_{k=1}^n {\bf
i}_k\varepsilon,{\bf i}\in{\bf Z}_+^n$. Thus
$$
\phi_{\bf ij}^{(\varepsilon)}(\lambda)\leq\phi_{\bf ij}(\lambda)
$$ for any
${\bf i},{\bf j}\in{\bf Z}_+^n$. Therefore, $p_{\bf
ij}^{(\varepsilon)}(t)\leq p_{\bf ij}(t)$. From the above,we know
$\lambda_Z^{(\varepsilon)}=-A(q_1^{(\varepsilon)},\cdots,q_1^{(\varepsilon)})$
is the decay parameter of $P^{(\varepsilon)}(t)$. Therefore, we have
$\lambda_Z\leq
\lambda_Z^{(\varepsilon)}=-A(q_1^{(\varepsilon)},\cdots,q_1^{(\varepsilon)})$.
Letting $\varepsilon\downarrow 0$ yields
$\lambda_Z\leq-A(1,\cdots,1)=-A(q_1,\cdots,q_n)$. \hfill$\Box$
\end{proof}

\par
Having given the decay parameter, we now consider the
$\lambda_Z$-invariant vectors/ measures and quasi-stationary
distribution. We first consider the $\lambda_Z$-invariant vectors.
From now on, we shall assume that $Q$ is conservative and ${\bf
Z}_+^n$ is communicating.
\begin{theorem}
\label{th6.3} Suppose that the $q$-matrix $Q$ as defined in
$(\ref{5.1})$--$(\ref{5.2})$, Let $P(t)=(p_{{\bf i j}}(t);{\bf
i,j}\in{\bf Z}_+^n)$ be the Feller minimal $Q$-function and
$\lambda_Z$ be the decay parameter of ${\bf Z}_+^n$. Then a
$\lambda_Z$-invariant vector $(y_{\bf j};{\bf j}\in {\bf Z}_+^n)$
for $Q$ {\rm{(}}or for $P(t)${\rm{)}} on ${\bf Z}_+^n$ is given by
\begin{eqnarray}\label{6.2}
 y_{\bf
j}=q_1^{j_1}\cdots q_n^{j_n},\ \ \ \ {\bf j}=(j_1,\cdots,j_n)\in
{\bf Z}_+^n
\end{eqnarray}
where $(q_1,\cdots,q_n)$ is the smallest nonnegative solution of
$(\ref{2.1})$.
\end{theorem}
\par
\begin{proof}
By Theorem~\ref{th6.1} we know that $(y_{\bf j};{\bf j}\in {\bf
Z}_+^n)$ is a $\lambda_Z$-invariant vector for $Q$ on ${\bf Z}_+^n$.
Therefore, it suffices to show that it is also $\lambda_Z$-invariant
for $P(t)$ on ${\bf Z}_+^n$. Indeed, by Proposition $5.4.1$ in
Anderson~\cite{-Anderson-p}, we know that for any ${\bf i}\in {\bf
Z}_+^n$ and $t\geq 0$,
\begin{eqnarray}\label{6.3}
 \sum_{{\bf j}\in {\bf
Z}_+^n\setminus {\bf 0}}p_{\bf ij}(t)y_{\bf j}\leq e^{-\lambda_Z
t}y_{\bf i}.
\end{eqnarray}
Hence, it follows from Kolmogorov forward equations that for ${\bf
i}\in {\bf Z}_+^n,\ (u_1,\cdots,u_n)\in [0,q_1]\times\cdots\times
[0,q_n]$,
\begin{eqnarray}\label{6.4}
\sum_{{\bf j}\in {\bf Z}_+^n}p'_{{\bf i}{\bf j}}(t)u_1^{j_1}\cdots
u_n^{j_n}&=&\sum_{j=1}^n B_j(u_1,\cdots,u_n)\sum_{{\bf k}\in {\bf
Z}_+^n\setminus {\bf 0}}p_{\bf ik}(t)k_j u_1^{k_1}\cdots
u_j^{k_j-1}\cdots u_n^{k_n} \nonumber \\
& &+A(u_1,\cdots,u_n)\sum_{{\bf k}\in {\bf Z}_+^n}p_{{\bf i}{\bf
k}}(t)u_1^{k_1}\cdots u_n^{k_n}
\end{eqnarray}
Therefore
$$
\aligned  \sum_{{\bf j}\in {\bf Z}_+^n}p_{{\bf i}{\bf
j}}(t)u_1^{j_1}\cdots u_n^{j_n}-u_1^{i_1}\cdots u_n^{i_n}
=&\sum_{j=1}^n B_j(u_1,\cdots,u_n)\int_0^t\sum_{{\bf k}\in {\bf
Z}_+^n\setminus {\bf 0}}p_{\bf ik}(s)k_j u_1^{k_1}\cdots
u_j^{k_j-1}\cdots u_n^{k_n}ds\\
&+A(u_1,\cdots,u_n)\int_0^t\sum_{{\bf k}\in {\bf Z}_+^n}p_{{\bf
i}{\bf k}}(s)u_1^{k_1}\cdots u_n^{k_n}ds
\endaligned
$$
Let $u_i=q_i(i=1,\cdots,n)$ in the above equation, we further have
\begin{eqnarray*}
\sum_{{\bf j}\in {\bf Z}_+^n}p_{{\bf i}{\bf j}}(t)q_1^{j_1}\cdots
q_n^{j_n}&-&q_1^{i_1}\cdots
q_n^{i_n}=A(q_1,\cdots,q_n)\int_0^t\sum_{{\bf k}\in {\bf
Z}_+^n}p_{{\bf i}{\bf k}}(s)q_1^{k_1}\cdots q_n^{k_n}ds
\end{eqnarray*}
i.e.
\begin{eqnarray}\label{6.5}
\sum_{{\bf j}\in {\bf Z}_+^n}p_{{\bf i}{\bf j}}(t)-y_{\bf
i}=-\lambda_{Z}\int_0^t\sum_{{\bf j}\in {\bf Z}_+^n}p_{{\bf
ij}}(s)y_{\bf j}ds.
\end{eqnarray}
Therefore
$$
\sum_{{\bf j}\in {\bf Z}_+^n}p_{{\bf ij}}(s)y_{\bf
j}=e^{-\lambda_{Z}t}y_{\bf i}.
$$
which implies that $(y_{\bf j};{\bf j}\in {\bf Z}_+^n)$ is a
$\lambda_Z$-invariant for $Q${\rm{(}}or for $P(t)${\rm{)}} on ${\bf
Z}_+^n$.\hfill$\Box$
\end{proof}
\par
The above Theorem gives a $\lambda_Z$-invariant vector for $Q$
{\rm{(}}or for $P(t)${\rm{)}} on ${\bf Z}_+^n$. We next consider the
$\lambda_Z$-invariant measures for $Q$ {\rm{(}}or for $P(t)${\rm{)}}
on ${\bf Z}_+^n$.

\par
\begin{theorem}
\label{th6.4} Suppose that $q$-matrix $Q$ defined in {\rm
(4.1)--(4.2)} is conservative, $G(1,\cdots,1)$ is positively regular
and $\{B_i(u_1,\cdots,u_n); 1\leq i\leq n\}$ is nonsingular. Let
$P(t)=(p_{{\bf i j}}(t);{\bf i,j}\in{\bf Z}_+^n)$ be the Feller
minimal $Q$-function and $\lambda_Z$ be the decay parameter of ${\bf
Z}_+^n$. Then for any $\lambda \in [0,\lambda_Z]$,
\par
{\rm (i)}\ There exists a $\lambda$-invariant measure $(m_{\bf
i};{\bf i}=(i_1,\cdots,i_n)\in {\bf Z}_+^n)$ for $Q$ on ${\bf
Z}_+^n$. Moreover, the generating function of this
$\lambda$-invariant measure $M(u_1,\cdots,u_n)=\sum_{{\bf i}\in {\bf
Z}_+^n}m_{\bf i}u_1^{i_1}\cdots u_n^{i_n}$ satisfies the following
partial differential equation
\begin{eqnarray}\label{6.6}
\sum_{i=1}^nB_i(u_1,\cdots,u_n)M_{u_i}(u_1,\cdots,u_n)+(\lambda+A(u_1,\cdots,u_n))
M(u_1,\cdots,u_n)=0.
 \end{eqnarray}
\par
{\rm (ii)}\ This measure $(m_{\bf i};{\bf i}=(i_1,\cdots,i_n)\in
{\bf Z}_+^n)$ is also $\lambda$-invariant for $P(t)$ on ${\bf
Z}_+^n$.
\par
{\rm (iii)}\ For $\lambda \leq \lambda_Z$, this $\lambda$-invariant
measure is convergent{\rm (}i.e., $\sum_{{\bf i}\in {\bf
Z}_+^n}m_{\bf i}<\infty${\rm )} if and only if $\lambda=\lambda_Z,\
\rho(1,\cdots,1)\leq 0$ and
$$
  \int_0^1\frac{\lambda+A(u,u_2(u),\cdots,u_n(u))}{B_1(u,u_2(u),\cdots,u_n(u))}du>-\infty
$$
where $u_k(u)\ (k=2,\cdots,n)$ are defined in Theorem~\ref{th3.1}.
\end{theorem}
\par
\begin{proof}
We first assume that $\lambda\in [0,\lambda_Z)$. It follows from
Kolmogorov forward equation that for any ${\bf i},{\bf j}\in {\bf
Z}_+^n$,
\begin{eqnarray*}
 p'_{{\bf i j}}(t)=\sum_{{\bf k}\in{\bf Z}_+^n}p_{{\bf i k}}(t)q_{{\bf k j}}.
\end{eqnarray*}
Therefore,
\begin{eqnarray}\label{6.7}
 \ \ \lambda \int_0^{\infty}e^{\lambda t}p_{{\bf i
0}}(t)dt+a_{\bf 0}\int_0^{\infty}e^{\lambda t}p_{{\bf i
0}}(t)dt+\sum_{j=1}^n b^{(j)}_{\bf 0}\int_0^{\infty}e^{\lambda
t}p_{{\bf i e}_j}(t)dt=-\delta_{\bf i 0},
\end{eqnarray}
and for ${\bf j}\in {\bf Z}_+^n\setminus{\bf 0}$,
\begin{eqnarray}\label{6.8}
 \ \lambda \int_0^{\infty}e^{\lambda
t}p_{{\bf i j}}(t)dt+\sum_{{\bf k}\in {\bf Z
}_+^n}(\int_0^{\infty}e^{\lambda t}p_{{\bf i k}}(t)dt)q_{{\bf k
j}}=-\delta_{{\bf i j}}.
\end{eqnarray}
Denote $m^{({\bf i})}_{\bf j}=(\int_0^{\infty}e^{\lambda t}p_{{\bf i
0}}(t)dt)^{-1}\cdot\int_0^{\infty}e^{\lambda t}p_{{\bf i j}}(t)dt$
and $\varepsilon^{({\bf i})}_{\bf j}=(\int_0^{\infty}e^{\lambda
t}p_{{\bf i 0}}(t)dt)^{-1}\cdot\delta_{\bf i j}$, then (\ref{6.7})
and (\ref{6.8}) can be rewritten as
\begin{eqnarray}\label{6.9}
(\lambda +a_{\bf 0})m^{({\bf i})}_{\bf 0}+\sum_{j=1}^n m^{({\bf
i})}_{{\bf e}_j}b^{(j)}_{\bf0}=-\varepsilon^{({\bf i})}_{{\bf 0}}
\end{eqnarray}
and for ${\bf j}\in {\bf Z}_+^n\setminus{\bf 0}$,
\begin{eqnarray}\label{6.10}
 \ \lambda m^{({\bf i})}_{\bf
j}+\sum_{{\bf k}\leq {\bf j}+{\bf 1}}m^{({\bf i})}_{\bf k}q_{{\bf k
j}}=-\varepsilon^{({\bf i})}_{{\bf j}}.
\end{eqnarray}
Let $H=\{l\geq 0;\ b^{(l)}_{\bf 0}=0\}\neq \emptyset$, then by the
irreducibility we know that
\par
(a)\ for any $l\in H$, there exists ${\bf k}$ such that $q_{{\bf
k}e_l}>0$ and ${\bf k}={\bf 0}$ or ${\bf k}=e_i$ for some $i\neq l$
or ${\bf k}=e_l+e_i$ for some $i\neq l$.
\par
(b)\ there exists ${\bf k}\in \{e_l;l\in H\}^c$ such that $q_{{\bf
k}e_l}>0$ for some $l\in H$.
\par
By (a), (b) and note that $m^{(\bf i)}_{\bf 0}=1$ and $m^{({\bf
i})}_{{\bf e}_ j}\geq 0 (j=1,\cdots,n)$, it can be seen from
(\ref{6.9}) and (\ref{6.10}) that there exist $(m_{\bf i};{\bf
i}=(i_1,\cdots,i_n)\in {\bf Z}_+^n)$ which is nonnegative and finite
such that
\begin{eqnarray}\label{6.11}
 \ \lambda m_{\bf
j}+\sum_{{\bf k}\leq {\bf j+1}}m_{\bf k}q_{{\bf k j}}=0,\ \ {\bf
j}\in {\bf Z}_+^n.
\end{eqnarray}
Now we claim that all $m_{\bf j}\ ({\bf j}\in {\bf Z}_+^n)$ are
positive. Indeed, note that $m_{{\bf 0}}>0$. If $m_{ \tilde{\bf
j}}=0$ for some $\tilde{{\bf j}}\in {\bf Z}_+^n$, then by the
irreducibility of ${\bf Z}_+^n$, we know that there exists ${\bf
j}_0={\bf 0}, {\bf j}_1, \cdots, {\bf j}_n=\tilde{{\bf j}}\in {\bf
Z}_+^n$ such that
$$
q_{{\bf j}_k {\bf j}_{k+1}}>0, \ \ \ k=0,\cdots, n-1.
$$
Hence by repeatedly using (\ref{6.11}) we know that $m_{{\bf 0}}=0$,
which is a contradiction. Therefore $(m_{\bf j};{\bf j}\in {\bf
Z}_+^n)$ is a $\lambda_Z$-invariant  measure for $Q$ on ${\bf
Z}_+^n$. By letting $\lambda\uparrow \lambda_Z$ in (\ref{6.11}) and
a similar argument as above, we get a $\lambda_Z$-invariant measure
for $Q$ on ${\bf Z}_+^n$.
\par
Since  $\lambda <-a_{\bf 0}$, multiplying $u_1^{j_1}\cdots
u_n^{j_n}$ on both sides of (\ref{6.11}) and summing over ${\bf
j}\in {\bf Z}_+^n$ yields that for $|u_1|,\cdots,|u_n|<(-a_{\bf
0}-\lambda)(\max\{b^{(i)}_{\bf 0};i=1,\cdots,n\})^{-1}$,
\begin{eqnarray}\label{6.12}
\sum_{i=1}^nB_i(u_1,\cdots,u_n)M_{u_i}(u_1,\cdots,u_n)+(\lambda+A(u_1,\cdots,u_n))
M(u_1,\cdots,u_n)=0
\end{eqnarray}
where $M(u_1,\cdots,u_n)=\sum_{{\bf j}\in {\bf Z}_+^n}m_{\bf
j}u_1^{j_1}\cdots u_n^{j_n}$. Since there exists
$(u_1,\cdots,u_n)\uparrow (q_1,\cdots,q_n)$ such that $
B_i(u_1,\cdots,u_n)>0$,
  it is easily seen that (\ref{6.12}) holds for $(u_1,\cdots,u_n)\in [0,q_1)\times\cdots\times
  [0,q_n)$. (i) is proved.
\par
Next, we prove (ii). Denote $g_{\bf j}(t)=\sum_{\bf i\in {\bf
Z}_+^n}m_{\bf i}p_{\bf ij}(t),\ {\bf j}\in {\bf Z}_+^n$. Then
$$
\aligned
   \sum_{{\bf k}\in {\bf Z}_+^n}g_{\bf k}(t)q_{\bf kj}&=\sum_{{\bf k}\in {\bf Z}_+^n}\sum_{{\bf i}\in {\bf Z}_+^n}
   m_{\bf i}p_{\bf ik}(t)q_{\bf kj}=\sum_{{\bf i}\in {\bf Z}_+^n}m_{\bf i}p'_{\bf ij}(t)\\
   &=\sum_{{\bf i}\in {\bf Z}_+^n}m_{\bf i}\sum_{{\bf k}\in {\bf Z}_+^n}q_{\bf ik}p_{\bf kj}(t)
   =\sum_{{\bf k}\in {\bf Z}_+^n}\sum_{{\bf i}\in {\bf Z}_+^n}m_{\bf i}q_{\bf ik}p_{\bf kj}(t)\\
   &=-\lambda
   g_{\bf j}(t)
   \endaligned
$$
and hence $(g_{\bf j}(t);{\bf j}\neq {\bf 0})$ is also a
$\lambda$-invariant measure for $Q$. On the other hand, it follows
from the Kolomogorov forward equation we have
$$
p_{\bf ij}(t)-\delta_{\bf ij}=\sum_{{\bf k}\in {\bf Z}_+^n}\int_0^t
p_{\bf ik}(u)du\cdot q_{\bf kj}
$$
Therefore,
$$
\aligned \sum_{{\bf i}\in {\bf Z}_+^n} m_{\bf i}|p_{\bf ij}(t+\Delta
t)-p_{\bf ij}(t)|&\leq \sum_{{\bf i}\in {\bf Z}_+^n} m_{\bf
i}\sum_{{\bf k}\leq {\bf j}+{\bf 1}}\int_t^{t+\Delta
t}p_{\bf ik}(u)du\cdot |q_{\bf kj}|\\
&=\sum_{{\bf k}\leq {\bf j}+{\bf 1}}\int_t^{t+\Delta t}\sum_{{\bf
i}\in {\bf Z}_+^n}
m_{\bf i}p_{\bf ik}(u)du\cdot |q_{\bf kj}|\\
&\leq\sum_{{\bf k}\leq {\bf j}+{\bf 1}}\int_t^{t+\Delta t}e^{-\lambda u} m_{\bf k}du\cdot|q_{\bf kj}|\\
&=\int_t^{t+\Delta t}e^{-\lambda u}du\cdot\sum_{{\bf k}\leq {\bf j}+{\bf 1}}m_{\bf k}|q_{\bf kj}|\\
&\rightarrow 0
\endaligned
$$
as $\Delta t\rightarrow 0$ since $\sum_{{\bf k}\leq {\bf j}+{\bf
1}}m_{\bf k}|q_{\bf kj}|$ is finite. Thus,
  $g_{\bf j}(t)=\sum_{{\bf i}\in {\bf
Z}_+^n}m_{\bf i}p_{\bf ij}(t)$ is a continuous function of $t\in
[0,\infty)$ and hence
$$
  \sum_{{\bf i}\in {\bf
Z}_+^n}m_{\bf i}p'_{\bf ij}(t)=\sum_{{\bf i}\in {\bf
Z}_+^n}\sum_{{\bf k}\leq {\bf j}+{\bf 1}}m_{\bf i}p_{\bf
ik}(t)q_{\bf kj}
$$
is also continuous. Therefore, by analysis theory we know that
$\sum_{{\bf i}\in {\bf Z}_+^n}m_{\bf i}p'_{\bf ij}(t)$ is uniformly
convergent on any bounded interval and hence
$$
g'_{\bf j}(t)=\sum_{{\bf i}\in {\bf Z}_+^n}m_{\bf i}p'_{\bf
ij}(t)=-\lambda
   g_{\bf j}(t),\ \ \forall t\geq 0
$$
which implying that
$$
  g_{\bf j}(t)=\sum_{{\bf i}\in {\bf
Z}_+^n}m_{\bf i}p_{\bf ij}(t)=g_{\bf j}(0)e^{-\lambda t}=m_{\bf
j}e^{-\lambda
  t}.
$$
Therefore, $(m_{\bf j};{{\bf j}\in {\bf Z}_+^n})$ is
$\lambda$-invariant for $P(t)$.
 \par
 Now we prove
(iii). Suppose that $\rho(1,\cdots,1)>0$. If $M(1,\cdots,1)<\infty$,
then $\lambda_Z=-A(q_1,\cdots,q_n)>0$ and (\ref{6.12}) holds for
$(u_1,\cdots,u_n)\in [0,1)^n$, furthermore,
\begin{eqnarray}
\label{6.13}
 \aligned
 \lim_{u_1\uparrow
1,\cdots,u_n\uparrow
1}\sum_{j=1}^nB_j(u_1,\cdots,u_n)M_{u_j}(u_1,\cdots,u_n)=0.
\endaligned
\end{eqnarray}
Letting $(u_1,\cdots,u_n)\uparrow(1,\cdots,1)$ in (\ref{6.12})
yields a contradiction.
 \par
Suppose that $\rho(1,\cdots,1)\leq 0$. By (\ref{6.12}) and
Theorem~\ref{th3.1}, we have
$$
 \frac{M'(u,u_2(u),\cdots,u_n(u))}{M(u,u_2(u),\cdots,u_n(u))}
 =-\frac{\lambda+A(u,u_2(u),\cdots,u_n(u))}{ B_1(u,u_2(u),\cdots,u_n(u))}
$$
where $u_k(u)\ (k=2,\cdots,n)$ are defined in Theorem~\ref{th3.1}
and $M'=\frac{dM}{du}$. Hence,
$$
  M(u,u_2(u),\cdots,u_n(u))=M_{\bf 0}
  e^{-\int_0^1\frac{\lambda+A(u,u_2(u),\cdots,u_n(u))}{ B_1(u,u_2(u),\cdots,u_n(u))}du}
$$
which implies the conclusion. \hfill$\Box$
\end{proof}
\par
Based on the $\lambda_Z$-invariant measure on ${\bf Z}_+^n$, we
finally present the quasi-stationary distributions for $P(t)$ on
${\bf Z}_+^n$.
\par
\begin{theorem}
\label{th6.5} Suppose that $q$-matrix $Q$ defined in {\rm
(4.1)--(4.2)} is conservative, $G(1,\cdots,1)$ is positively regular
and $\{B_i(u_1,\cdots,u_n); 1\leq i\leq n\}$ is nonsingular. Let
$P(t)=(p_{{\bf i j}}(t);{\bf i,j}\in{\bf Z}_+^n)$ be the Feller
minimal $Q$-function and $\lambda_Z$ be the decay parameter of ${\bf
Z}_+^n$. Then there exists a quasi-stationary distribution for
$P(t)$ on ${\bf Z}_+^n$ if and only if $\rho(1,\cdots,1)\leq 0$ and
$$
\int_0^1\frac{\lambda+A(u,u_2(u),\cdots,u_n(u))}{
B_1(u,u_2(u),\cdots,u_n(u))}du>-\infty.
$$
Moreover, if these conditions hold, then the quasi-stationary
distribution $\{(m_{\bf i};{\bf i}\in {\bf Z}_+^n)\}$ satisfies the
equation (\ref{6.6}) with $\lambda=\lambda_Z$.
\end{theorem}
\par
\begin{proof}
By Proposition $3.1$ of Nair $\&$ Pollett$^{[11]}$, a probability
distribution $(m_{\bf i};{\bf i}\in {\bf Z}_+^n)$ is a
quasi-stationary distribution for $P(t)$ on ${\bf Z}_+^n$ if and
only if for some $\lambda>0$, $(m_{\bf i};{\bf i}\in {\bf Z}_+^n)$
is $\lambda$-invariant for $P(t)$ on ${\bf Z}_+^n$. Thus the
conclusions follow from Theorem~\ref{th6.4}.\hfill$\Box$
\end{proof}

\acks This work is supported by the National Natural Sciences
Foundation of China (,No.11371374, No.11571372), Research Fund for
the Doctoral Program of Higher Education of China
(No.20110162110060), the Graduate degree thesis Innovation
Foundation of Hunan Province (No.CX2011B077).

\end{document}